\documentclass [a4paper,oneside,10pt,reqno,final]{amsart}

\usepackage{amsmath, amssymb, amsthm}
\usepackage{enumitem}
\usepackage[all]{xy}

\usepackage{version,color}
\usepackage{showkeys}
\usepackage[l2tabu,orthodox]{nag}
\usepackage[all, warning]{onlyamsmath}

\numberwithin{equation}{section}
\numberwithin{figure}{section}
\allowdisplaybreaks

\theoremstyle{definition}
\newtheorem{thm}{Theorem}[section]
\newtheorem{prp}[thm]{Proposition}
\newtheorem{lem}[thm]{Lemma}
\newtheorem{cor}[thm]{Corollary}
\newtheorem{dfn}[thm]{Definition}

\newtheorem{rmk}[thm]{Remark}
\newtheorem*{thm*}{Theorem}
\newtheorem*{prp*}{Proposition}
\newtheorem*{lem*}{Lemma}
\newtheorem*{dfn*}{Definition}
\newtheorem*{fct*}{Fact}
\newtheorem*{rmk*}{Remark}

\newcommand{\ol}{\overline}
\newcommand{\wh}{\widehat}
\newcommand{\wt}{\widetilde}
\newcommand{\inj}{\hookrightarrow}
\newcommand{\longto}{\longrightarrow}
\newcommand{\simto}{\xrightarrow{\sim}}
\newcommand{\longsimto}{\xrightarrow{\, \sim \, }}
\newcommand{\longinj}{\lhook\joinrel\longrightarrow}
\newcommand{\longsurj}{\relbar\joinrel\twoheadrightarrow}
\newcommand{\defiff}{\stackrel{\text{def}}{\iff}}
\newcommand{\wotimes}{\mathbin{\wh{\otimes}}}
\newcommand{\id}{\mathrm{id}}

\newcommand{\nil}{\text{nil}}

\newcommand{\tpar}{\text{partitions}}

\newcommand{\bbA}{\mathbb{A}}
\newcommand{\bbC}{\mathbb{C}}
\newcommand{\bbF}{\mathbb{F}}
\newcommand{\bbN}{\mathbb{N}}

\newcommand{\bbZ}{\mathbb{Z}}

\newcommand{\calF}{\mathcal{F}}
\newcommand{\calG}{\mathcal{G}}
\newcommand{\calH}{\mathcal{H}}

\newcommand{\catA}{\mathfrak{A}}
\newcommand{\catP}{\mathfrak{P}}
\newcommand{\Hcl}{\mathrm{H}_{\text{cl}}}
\newcommand{\Dcl}{\mathrm{DH}_{\text{cl}}}
\newcommand{\wDcl}{\widehat{\mathrm{DH}}_{\text{cl}}}
\newcommand{\Res}{\mathop{\mathrm{Res}}}

\DeclareMathOperator{\F}{F}
\DeclareMathOperator{\R}{R}
\DeclareMathOperator{\Dh}{DH}
\DeclareMathOperator{\GL}{GL}
\DeclareMathOperator{\Gr}{Gr}
\DeclareMathOperator{\Ho}{Ho}
\DeclareMathOperator{\Ob}{Ob}
\DeclareMathOperator{\Aut}{Aut}

\DeclareMathOperator{\End}{End}
\DeclareMathOperator{\Ext}{Ext}
\DeclareMathOperator{\Hom}{Hom}

\DeclareMathOperator{\Iso}{Iso}
\DeclareMathOperator{\Ker}{Ker}
\DeclareMathOperator{\Rep}{\mathfrak{Rep}}
\DeclareMathOperator{\catmod}{\mathfrak{mod}}

\newcommand{\abs}[1]{\left| #1 \right|}
\newcommand{\EF}[2]{\chi\left(#1, #2 \right)}
\newcommand{\sEF}[2]{q^{\chi\left(#1, #2 \right)/2}}
\newcommand{\pair}[2]{\left< #1, #2 \right>}
\newcommand{\bnm}[2]{\genfrac{[}{]}{0pt}{1}{#1}{#2}}

\begin{document}

\title{A study of symmetric functions via derived Hall algebra}

\author{Ryosuke Shimoji, Shintarou Yanagida}
\address{Graduate School of Mathematics, Nagoya University, 
Furocho, Chikusaku, Nagoya, Japan, 464-8602.}
\email{m17017f@math.nagoya-u.ac.jp, 
yanagida@math.nagoya-u.ac.jp}

\date{December 15, 2018}

\begin{abstract}
We use derived Hall algebra 
of the category of nilpotent representations 
of Jordan quiver to reconstruct the theory of symmetric functions,
focusing on Hall-Littlewood symmetric functions and 
various operators acting on them. 
\end{abstract}

\maketitle

\setcounter{section}{-1}
\section{Introduction}

In this note we use derived Hall algebra 
of the category $\Rep^{\nil}_{\bbF_q} Q$ 
of nilpotent representations of Jordan quiver $Q$ to
reconstruct the theory of symmetric functions,
focusing on Hall-Littlewood symmetric functions and 
various operators acting on them. 

As is well known, the ring $\Lambda$ of symmetric functions
has a structure of Hopf algebra,
and it coincides with the classical Hall algebra $\Hcl$, i.e., 
the Ringel-Hall algebra of the category $\Rep^{\nil}_{\bbF_q} Q$.
In particular, the Hall-Littlewood symmetric function 
$P_\lambda(x;q^{-1}) \in \Lambda$ of parameter $q^{-1}$ 
corresponds to the isomorphism class $[I_\lambda]$ 
of the nilpotent representation $I_\lambda$ 
constructed from the Jordan matrix 
associated to the partition $\lambda$.
See Macdonald's book \cite[Chap.I\!I, I\!I\!I]{M} and 
Schiffmann's lecture note \cite[\S2]{S} for detailed account.

Thus, in principle, one may 
\emph{reconstruct the theory of symmetric functions
only using the knowledge of the classical Hall algebra} $\Hcl$,
without any knowledge of symmetric functions nor 
the structure of $\Lambda$.
One of our motivation is pursue this strategy.

Recall that in the theory of symmetric functions
we have some key ingredients to construct 
important bases of $\Lambda$. 
\begin{itemize}[nosep]
\item
\emph{Power-sum symmetric functions} 
$p_n(x)=\sum_i x_i^n$.

\item
\emph{Inner products} $\pair{\cdot}{\cdot}$ on $\Lambda$.

\item
\emph{Kernel functions} associated to inner products 
and operators on $\Lambda$.
\end{itemize}
Some explanation on the second item is in order.
In \cite[Chap.V\!I]{M} Macdonald introduced 
his two-parameter symmetric functions $P_\lambda(x;q,t) \in \Lambda$.
His discussion stars with the inner product 
$\pair{p_m(x)}{p_n(x)}_{q,t}=\delta_{m,n} n (1-q^n)/(1-t^n)$,
and in the end, $P_\lambda(x;q,t)$'s 
give an orthogonal basis of $\Lambda$.
The Hall-Littlewood symmetric function $P_\lambda(x;t)$
and the Schur symmetric functions $s_\lambda(x)$ 
can be obtained from $P_\lambda(x;q,t)$ by degeneration of parameters,
and they are orthogonal bases with respect to the inner product
$\pair{p_m(x)}{p_n(x)}_{t}=\delta_{m,n} n/(1-t^n)$
and $\pair{p_m(x)}{p_n(x)}=\delta_{m,n} n$
respectively.

Viewing these key ingredients, we infer that 
what should be discussed first is 
the realization of power-sum symmetric functions purely 
in terms of the classical Hall algebra $\Hcl$.
Such a realization is already known 
(for example see \cite[Chap.I\!I \S7 Example 2]{M}),
but has not been much stressed as far as we know.

\begin{dfn*}[Definition \ref{dfn:p_n}]
We define $p_n \in \Hcl$ by 
\[
 p_n := \sum_{\abs{\lambda}=n} 
 (q;q)_{\ell(\lambda)-1} \cdot [I_\lambda]. 
\]
\end{dfn*}

Here the summation is over the partitions 
$\lambda=(\lambda_1,\lambda_2,\ldots)$
with $\abs{\lambda}:=\sum_i \lambda_i$ equal to $n$.
$\ell(\lambda)$ means the length of the partition $\lambda$,
and $(q;q)_l := \prod_{i=1}^l(1-q^i)$.

The logic line of our argument is as follows:
We introduce $p_n$ by this formula,
and only use the structure of $\Hcl$ to deduce

\begin{thm*}[{Theorem \ref{thm:prim}}]
$p_n$ is a primitive element of $\Hcl$, i.e., 
$\Delta(p_n)=p_n \otimes 1 + 1 \otimes p_n$ 
in terms of the Green coproduct $\Delta$.
\end{thm*}

\begin{thm*}[{Theorem \ref{thm:p_n:pair}}]
In terms of Green's Hopf pairing $\pair{\cdot}{\cdot}$ on $\Hcl$, 
we have $\pair{p_m}{p_n} = \delta_{m,n} n/(q^n-1)$.
\end{thm*}

As a consequence, we have two orthogonal bases 
$[I_\lambda]$ and $p_\lambda$ on $\Hcl$ 
with respect to $\pair{\cdot}{\cdot}$.
Then we recover the \emph{Cauchy-type kernel function},
which is the third key ingredient.

\begin{prp*}[{Proposition \ref{prp:Cauchy}}]
Set elements of $\Hcl$ as 
$P_\lambda := q^{n(\lambda)} [I_\lambda]$ and 
$Q_\lambda := P_\lambda/\pair{P_\lambda}{P_\lambda}$.
Then
\[
 \sum_{\lambda: \tpar} P_{\lambda} \otimes Q_{\lambda}
=\exp\Bigl(\sum_{n\ge 1} \frac{1}{n} (q^n-1) p_n \otimes p_n\Bigr).
\]
\end{prp*}

We will call $P_\lambda$ the Hall-Littlewood element
(Definition \ref{dfn:P_lambda}),
since under the isomorphism $\psi:\Hcl \simto \Lambda$ 
(Theorem \ref{thm:Hall=Lambda})
it corresponds to the Hall-Littlewood symmetric function
$P_\lambda(x;q^{-1}) \in \Lambda$.

So far all the discussion is on the level of $\Hcl$,
i.e., the Ringel-Hall algebra of 
the abelian category $\Rep^{\nil}_{\bbF_q} Q$.
Our second motivation is to realize operators 
acting nicely on $P_\lambda$ in the framework of Hall algebra.
For this purpose we need To\"{e}n's derived Hall algebra \cite{T}.

At present we know various nice operators 
in the theory of symmetric functions.
The main ingredient to construct such operators is 
\begin{itemize}
\item
the \emph{differential operator} $\partial_{p_n(x)}$ 
in terms of the power-sum function $p_n(x)$.
\end{itemize}
Thus it maybe nice to realize such differential operator 
in terms of Hall algebra.
The answer is hidden in the description of the derived Hall algebra 
for hereditary abelian category \cite[\S7]{T}.
Let us denote by $Z_\lambda^{[n]}$ 
the generator of the derived Hall algebra $\Dcl$
of the category $\Rep^{\nil}_{\bbF_q} Q$
(see \S\ref{subsec:Heis} for the detail).

\begin{thm*}[{Theorem \ref{thm:Heis}}]
For $n \in \bbZ_{>0}$ define $b_{\pm n} \in \Dcl$ by 
\[
 b_n := 
 \sum_{\abs{\lambda}=n}(q;q)_{\ell(\lambda)-1} Z_\lambda^{[0]},
 \quad
 b_{-n} := 
 \sum_{\abs{\lambda}=n}(q;q)_{\ell(\lambda)-1} Z_\lambda^{[1]},
\]
and set $b_0 := 1 \in \Dcl$.
Then for $m,n \in \bbZ$ we have 
\[
 b_m * b_n - b_n * b_m = \delta_{m+n,0}\frac{m}{q^m-1}.
\]
\end{thm*}

This Heisenberg relation means that 
identifying $[I_\lambda] \in \Hcl$ with $Z_\lambda^{[1]} \in \Dcl$,
we realize $\partial_{p_n}$ as $b_n \in \Dcl$.

As an application of this Heisenberg subalgebra,
we study certain vertex operator 
whose zero mode has $P_\lambda \in \Hcl \subset \Dcl$
as eigenfunctions.
See \S \ref{subsec:eigen} for the detail.

\subsection*{Notations}

We denote by $\bbN:=\{0,1,2,\ldots\}$ the set of non-negative integers,
and by $\abs{S}$ the cardinality of a set $S$.
For a category $\catA$, the class of objects is denoted by $\Ob(A)$.
For $M,N \in \Ob(A)$,
the set of morphisms from $M$ to $N$ is denoted by $\Hom_{\catA}(M,N)$.

The tensor symbol $\otimes$ means the one over the complex number field $\bbC$
unless otherwise stated.

We follow Macdonald \cite[Chap.1 \S1]{M} as for notations of partitions.
A partition means a non-increasing sequence  
$\lambda=(\lambda_1,\ldots,\lambda_n)$
of non-negative integers of finite length.
We identify a partition and the one appended with $0$'s, so 
$\lambda=(\lambda_1,\ldots,\lambda_n)
=(\lambda_1,\ldots,\lambda_n,0,\ldots)$.
For a partition $\lambda$, we set 
\[
 \abs{\lambda}:=\sum_{i \ge 1}\lambda_i, \quad
 n(\lambda) := \sum_{i \ge 1}(i-1)\lambda_i,
\]
and denote by $\ell(\lambda)$
the largest index $i$ such that $\lambda_i>0$.
We sometimes call $\ell(\lambda)$ the length of $\lambda$.
For a partition $\lambda$ and $j \in \bbZ_{>0}$, 
we set $m_j(\lambda) := \{i \in \bbZ_{>0} \mid \lambda_i = j\}$,
and express $\lambda=(1^{m_1(\lambda)},2^{m_2(\lambda)},\ldots)$.
Finally, the transpose of a partition $\lambda$ is denoted by $\lambda'$.

Let $x$ and $q$ be indeterminates.
For $n \in \bbZ_{\le 0}$ we set $(x;q)_n:=1$,
and for $n \in \bbZ_{>0}$ we set
\[
 (x;q)_n := \prod_{i=1}^n (1- x q^{i-1}).
\]

\subsection*{Acknowledgements.} 

This note is based on the master thesis of the first author.
The second author is supported by the Grant-in-aid for 
Scientific Research (No.\ 16K17570), JSPS.
This work is also supported by the JSPS Bilateral Program
``Topological Field Theory and String Theory -- 
  from topological recursion to quantum toroidal algebras".

\newpage

\section{Ringel-Hall algebra and To\"{e}n's derived Hall algebra}
\label{sec:general}

In this section we give a brief account on the Ringel-Hall algebra
and the derived Hall algebra, based on the original papers 
of Ringel \cite{R}, Green \cite{G}, Xiao \cite{X} 
and To\"{e}n \cite{T},
and also on the lecture note of Schiffmann \cite[\S1]{S}.

\subsection{Ringel-Hall algebra}
\label{subsec:RH}

We call a category $\catA$ essentially small
if the isomorphism classes of objects form a set,
which is denoted by $\Iso(\catA)$.
For an object $A$ of $\catA$ 
its isomorphism class is denoted by $[A] \in \Iso(\catA)$.

For an essentially small abelian category $\catA$ 
we denote by $K_0(\catA)$ the Grothendieck group,
and for an object $A \in \Ob(\catA)$ the associated element of $K_0(\catA)$ 
is denoted by $\ol{A}$.
So, for a short exact sequence $0 \to A \to B \to C \to 0$ in $\catA$,
we have $\ol{A}-\ol{B}+\ol{C}=0$.

Let $k=\bbF_q$ be a finite field with $\abs{k}=q$.
Let $\catA$ be a category satisfying the following conditions.
\begin{enumerate}[nosep, label=(\roman*)]
\item
Essentially small, abelian and $k$-linear.

\item
Of finite global dimension.
\end{enumerate}

We denote by $\F(\catA)$ 
the linear space of $\bbC$-valued functions on $\Iso(\catA)$
with finite supports. 
We have a basis $\{1_{[M]} \mid [M] \in \Iso(\catA)\}$ of $\F(\catA)$,
where $1_{[M]}$ means 
the characteristic function of $[M] \in \Iso(\catA)$.
The correspondence $1_{[M]} \mapsto [M]$ 
gives an identification
$\F(\catA) \simto \bigoplus_{[M] \in \Iso(\catA)}\bbC[M]$,
and we will always identify these two spaces.

For $A, B \in \Ob(\catA)$ we set 
$\EF{A}{B} := \sum_{i\ge0} (-1)^i \dim_k \Ext_{\catA}^i(A,B)$,
which depends only on $\ol{A},\ol{B} \in K_0(\catA)$.
The obtained map 
$\EF{\cdot}{\cdot}: K_0(\catA) \otimes_{\bbZ} K_0(\catA) \longto \bbZ$
is bilinear and called the Euler pairing.

For $M \in \Ob(\catA)$, we define $a_M \in \bbZ_{>0}$ by
\begin{align*}
a_M := \abs{\Aut(M)},
\end{align*}
and for $M,N,R \in \Ob(\catA)$, we set 
\begin{align*} 
&\Ext^1_{\catA}(M,N)_R := 
 \{0 \to N \to R \to M  \to 0 \mid \text{exact in } \catA\},
\\
&e_{M,N}^R :=  \abs{\Ext^1_{\catA}(M,N)_R}, \quad 
 g_{M,N}^R :=  a_M^{-1} a_N^{-1} e_{M,N}^R.
\end{align*}
Note that $a_M,e_{M,N}^R,g_{M,N}^R$ depend only on 
the isomorphism classes $[M],[N],[R] \in \Iso(\catA)$.

\begin{fct*}[{Ringel \cite{R}}]
For $[M], [N] \in \Iso(\catA)$ we set 
\[
 [M] * [N] := \sEF{M}{N} \sum_{[R] \in \Iso(\catA)} g^R_{M,N} [R],
\]
where we choose representatives $M,N,R \in \Ob(\catA)$ 
for the fixed isomorphism classes $[M],[N],[R] \in \Iso(\catA)$.
Denote by $[0] \in \Iso(\catA)$ the isomorphism class 
of the zero object $0$ in $\catA$.
Then $\bigl(\F(\catA),*,[0]\bigr)$ is an associative $\bbC$-algebra 
with $[0]$ the unit, which has a $K_0(\catA)$-grading.\
\end{fct*}

Let us recall the following another definition of $g_{M,N}^R$.

\begin{lem*}
For $M,N,R \in \Ob(\catA)$ we have
\[
 g^{R}_{M,N} = \abs{\calG^{R}_{M,N}}, \quad 
 \calG^{R}_{M,N} := 
 \{N' \subset R \mid N' \simeq N, \ R/N \simeq M\}.
\]
\end{lem*}

One can prove this statement by considering a free action 
of $\Aut(M) \times \Aut(N)$ on $\Ext^1_{\catA}(M,N)_R$.

Similarly the multi-component product has the following meaning.

\begin{lem}\label{lem:filter}
For $B_1,B_2,\ldots,B_r \in \Ob(\catA)$ we set 
\[
 \calF(A;B_1,B_2,\ldots,B_r) := 
 \{A=A_1 \supset A_2 \supset \cdots \supset A_r \supset A_{r+1} =0 
  \mid  A_i/A_{i+1} \simeq B_i \ (i=1,\ldots,r) \}.
\]
Then we have 
\[
 [B_1] * [B_2] * \cdots * [B_r] = 
 \sum_{[A] \in \Iso(\catA)} 
 \abs{\calF(A;B_1,\ldots,B_r)} \cdot [A].
\]
\end{lem}

Next we recall the coproduct on $\F(\catA)$.
As mentioned in Notations, we simply denote $\otimes := \otimes_\bbC$.

\begin{fct*}[{Green \cite{G}}]
Assume that the category $\catA$ satisfies the conditions (i), (ii) and
\begin{enumerate}[nosep, label=(\roman*)]
\setcounter{enumi}{2}
\item
Each object has only finite numbers of subobjects.
\end{enumerate}
Then we have a $K_0(\catA)$-graded coassociative $\bbC$-coalgebra 
$\bigl(\F(\catA),\Delta,\epsilon\bigr)$,
where the coproduct 
$\Delta: \F(\catA)\to \F(\catA) \otimes \F(\catA)$
and the counit $\epsilon: \R(\catA) \to \bbC$ are given by
\begin{align*}
 \Delta([R]) := 
 \sum_{[M],[N]} \sEF{M}{N} \frac{e_{M,N}^R}{a_R} [M] \otimes [N],
 \quad
 \epsilon([M]) := \delta_{M,0}.
\end{align*}
\end{fct*}

The coassociativity is a direct consequence of 
the associativity of the product $*$.

If the condition (iii) is not satisfied,
then the coproduct $\Delta$ is an infinite summation 
so that  $\bigl(\F(\catA),\Delta,\epsilon\bigr)$
is not a genuine coalgebra.
However, one can consider it as a topological coalgebra.
See \cite[\S1.4]{S} for the detail.

For the data $(\F(\catA),*,\Delta)$ to be a bialgebra,
we need another condition on the category $\catA$.
This fact is revealed by Green \cite{G}.

\begin{fct*}[{Green \cite{G}}]
Assume that the category $\catA$ satisfies the conditions (i),  (iii) and
\begin{enumerate}[nosep, label=(\roman*)]
\setcounter{enumi}{3}
\item
hereditary, i.e.,
the global dimension is $0$ or $1$
(so the condition (ii) is also satisfied).
\end{enumerate}
Define the product $*$ on $\F(A) \otimes \F(A)$ by 
\begin{align}\label{eq:*:tensor:twist}
 (x_1 \otimes x_2) * (y_1 \otimes y_2) := 
 \EF{x_2}{y_1}\EF{y_1}{x_2} \cdot (x_1 * x_2) \otimes (y_1 * y_2).
\end{align}
Then $\Delta: (\F(\catA),*) \to (\F(\catA)\otimes \F(\catA),*)$
is an algebra homomorphism respecting the $K_0(\catA)$-gradings.

Moreover, the bilinear pairing $\pair{\cdot}{\cdot}$
on $\F(\catA)$ given by 
\[
 \pair{[M]}{[N]} := \delta_{[M],[N]} a_M^{-1}
\]
is a Hopf pairing with respect to $*$ and $\Delta$.
That is, we have 
$\pair{x*y}{z}=\sum \pair{x}{z_1}\pair{y}{z_2}$
for any $x,y,z \in \F(\catA)$,
where we used the Sweedler notation $\Delta(z)=\sum z_1 \otimes z_2$.
\end{fct*}

Although the Ringel-Hall algebra is not a bialgebra 
in the standard sense since we use 
the twisted product \eqref{eq:*:tensor:twist},
we will say that 
$(\F(\catA),*,[0],\Delta,\epsilon)$ is a bialgebra 
in the above sense.

By the work of Xiao \cite{X},
$\R(\catA)$ has a structure of Hopf algebra. 

\begin{fct*}[{Xiao \cite{X}}]
Assume the conditions (i), (iii) and (iv) for $\catA$. 
Define a linear map $S:\F(\catA) \to \F(\catA)$ by 
\[
 S([M]) := 
 a_M^{-1}\sum_{r \ge 1}(-1)^r\sum_{M_{\bullet}\in \calF(M;r)}
 \bigl(\prod^r_{i=1} \EF{M_i/M_{i+1}}{M_{i+1}} a_{M_i/M_{i+1}}\bigr)
 \cdot [M_1/M_2] * [M_2/M_3] * \cdots * [M_r],
\]
where $\calF(M;r)$ denotes the set of filtrations of proper length $r$:
\[
 \calF(M;r) :=
 \{ M_{\bullet}= 
   (M=M_1\supsetneq M_2 \supsetneq \cdots \supsetneq M_r \supsetneq 0)\}.
\]
Then $(\F(\catA),*,[0],\Delta,\epsilon,S)$ is a 
$K_0(\catA)$-graded Hopf algebra over $\bbC$.
\end{fct*}

Thus we have a $K_0(\catA)$-graded Hopf algebra with Hopf pairing
\[
 \R(\catA) := 
 \bigl(\F(\catA),*,[0],\Delta,\epsilon,S,\pair{\cdot}{\cdot}\bigr),
\]
which we call the \emph{Ringel-Hall algebra} of the category $\catA$.

\subsection{dg Hall algebra}
\label{subsec:dg}

In this subsection we recall the description of
derived Hall algebra of hereditary abelian category
due to To\"{e}n \cite[\S7]{T}.

Let $\catA$ be a category satisfying the conditions 
(i) and (iv), i.e.,
essentially small hereditary abelian category 
linear over $\bbF_q$.
Thus we have 
the Ringel-Hall algebra $\R(\catA)=(\F(\catA),*,[0])$.
It is also equipped with Green's topological coproduct 
and Green's Hopf pairing, but we will not treat them.

Consider the dg category $\catP$ of perfect complexes 
consisting of objects in $\catA$.
By \cite[\S3]{T} we have a unital associative $\bbC$-algebra 
$\Dh(\catA)$ whose underlying linear space is 
spanned by the set of isomorphism classes of objects in $\catP$,
i.e., perfect complexes of $\catA$.

We denote by $\Ho(\catP)$ the associated homotopy category 
in terms of the model structure given in \cite{T07}.
By the assumption on $\catA$ we have an equivalence 
$D^b(\catA) \to H(\catP)$ of triangulated categories,
where the source means the bounded derived category of $\catA$.
Then we can apply the argument in \cite[\S7]{T}, 
and have the following description of $\Dh(\catA)$.

\begin{fct*}[{\cite[Proposition 7.1]{T}}]
The algebra $\Dh(\catA)$ is isomorphic to 
the unital associative algebra generated by 
$\{Z^{[n]}_x \mid x \in \Iso(\catA), n \in \bbZ\}$
and the following relations.
\begin{align*}
 Z_x^{[n]} * Z_y^{[n]} &= 
 \sum_{z \in \Iso(\catA)} g^{z}_{x,y} Z_x^{[n]},
\\
 Z_x^{[n]} * Z_y^{[n+1]} &= \sum_{k,c \in \Iso(\catA)} 
 q^{-\EF{c}{k}} \gamma^{k,c}_{x,y} Z_k^{[n+1]} * Z_c^{[n]},
\\
 Z_x^{[n]} * Z_y^{[m]} &= 
 q^{(-1)^{n-m}\EF{x}{y}} Z_y^{[m]} * Z_x^{[n]}
 \quad (n-m<-1).
\end{align*}
Here $g^z_{x,y}$ denotes the structure constants 
of the Ringel-Hall algebra $\R(\catA)$, and 
\[
 \gamma^{k,c}_{x,y} := 
 \frac{\abs{\{0 \to k \to y \to x \to c \to 0 
              \mid \text{exact in $\catA$}\}}}{a_x a_y}.
\]
\end{fct*}

\section{Classical Hall algebra}
\label{sec:Hcl}

In this subsection we recall basic properties of the Ringel-Hall algebra
of the category of nilpotent representations of 
the Jordan quiver over a finite field.
It coincides with the commutative algebra introduced by Steinitz and 
Hall in their study of the representation theory of symmetric group,
and is called the classical Hall algebra.
The main ingredient in this subsection is the structure theorem
(Theorem \ref{thm:Hall-Jordan}) of the classical Hall algebra.
Our presentation is based on \cite[\S2]{S} and \cite[Chap.\ I\!I]{M}.

\subsection{Category of nilpotent representation of Jordan quiver}
\label{subsec:Jordan}

Let $Q=(Q_0,Q_1)$ be the Jordan quiver
consisting of one vertex $Q_0=\{\bullet\}$
and one edge arrow  $Q_1=\{a\}$.
In this and next sections we only consider the category
$\catA = \Rep_k^{\nil} Q$
of nilpotent representation of $Q$ over a field $k$.

An object of $\catA$ is a pair $(V,x)$
of finite dimensional $k$-linear space $V$ and 
an endomorphism  $x \in \End_k(V)$.
$\catA$ is equivalent to the category $\catmod^{\nil} k[t]$
of modules over 
the polynomial ring $k[t]$ which are finite dimensional over $k$
and where the action of $t$ is nilpotent.

The category $\catA$ over a finite field $k=\bbF_q$
satisfies all the conditions (i)--(iv) in the last subsection,
so that we have the Ringel-Hall algebra $\R(\catA)$.
We call it the \emph{classical Hall algebra} and denote it by $\Hcl$. 

Let us describe the structure of the classical Hall algebra.
For $n \in \bbZ_{>0}$ we denote by $J_n$
the Jordan matrix of dimension $n$ with $0$ diagonal entries:
\[
 J_n := 
 \begin{bmatrix}
  0 & 1 & \\
    & 0 & 1 & \\
    &   & \ddots & \ddots \\
    &   &   & 0  & 1 \\
    &   &   &    & 0
 \end{bmatrix}.
\]
For a partition $\lambda=(\lambda_1, \lambda_2, \ldots)$ 
we set an object $I_\lambda$ of $\catA$ by 
\[
 I_\lambda := (k^{\abs{\lambda}},J_\lambda), \quad
 J_\lambda := J_{\lambda_1} \oplus J_{\lambda_2} \oplus \cdots.
\]
Here we used $|\lambda| := \sum_{i\ge1}\lambda_i$.
We also consider $\emptyset=(0)$ as a partition,
and set $I_\emptyset = (0,0)$.

Using $\catA \simeq \catmod^{\nil} k[t]$ one can deduce 

\begin{lem*}
Consider $\catA$ over an arbitrary field $k$.
\begin{enumerate}[nosep, label=(\arabic*)]
\item
$\Iso(\catA) = \{[I_\lambda] \mid \lambda: \tpar \}$.

\item
Simple objects of $\catA$ are isomorphic to $I_{(1)}=(k,0)$. 
Indecomposable objects are isomorphic to 
$I_{(n)}$ with some $n \in \bbN$.
\end{enumerate} 
\end{lem*}

Thus $\Hcl$ has a basis parametrized by partitions.

Recall that $\Hcl$ has a $K_0(\catA)$-grading.
Since there is a short exact sequence
$0 \to I_{(m)} \to I_{(l)} \to I_{(l)}/I_{(m)} \simeq I_{(l-m)} \to 0$
in $\catA$, we deduce

\begin{lem*}
Over an arbitrary field $k$, 
we have an isomorphism $K_0(\catA) \simto \bbZ$ of modules 
given by $\ol{I_\lambda} \mapsto \abs{\lambda}$.
\end{lem*}

Since $\Hom_{\catA}(I_{(1)},I_{(1)})=k$ 
and $\Ext^1_{\catA}(I_{(1)},I_{(1)})=k$,
we also have

\begin{lem*}
$\EF{\alpha}{\beta}=0$ for any $\alpha, \beta \in K_0(\catA)$.
\end{lem*}

Using these lemmas, one can write down the structure of 
the classical Hall algebra $\Hcl$ as 
\begin{align}\label{eq:Hall(A)}
\begin{split}
&\Hcl 
=\bigl(\F(\catA),*,[0],\Delta,\epsilon,S\bigr), \quad
 \F(\catA) = \oplus_{\lambda: \tpar}\bbC[I_\lambda],
\\
&[I_\mu]*[I_\nu] = 
 \sum_{\lambda: \tpar} g^{\lambda}_{\mu,\nu} [I_\lambda],
 \quad
 g^{\lambda}_{\mu,\nu} := \abs{\calG^{\lambda}_{\mu,\nu}},
 \quad
 \calG^{\lambda}_{\mu,\nu} := 
 \calG^{I_\lambda}_{I_\mu, I_\nu} = 
 \{N \subset I_\lambda \mid 
   N \simeq I_{\nu}, I_\lambda /N \simeq I_\mu\},
\\
&\Delta([I_\lambda]) := \sum_{\mu, \nu} 
 a_\lambda^{-1} a_\mu a_\nu g^\lambda_{\mu,\nu} 
 \cdot [I_\mu] \otimes [I_\nu],
 \quad
 a_\lambda :=a_{I_\lambda} = \abs{\Aut_{\catA}(I_\lambda)}.
\end{split}
\end{align}

The grading by $K_0(\catA)=\bbZ$ can be restated as 
\begin{lem}\label{lem:graded}
For partitions $\lambda,\mu,\nu$ 
with $\abs{\lambda} \neq \abs{\mu}+\abs{\nu}$,
we have 
$g^{\lambda}_{\mu,\nu} = 0$.
\end{lem}

We now recall the well-known theorem of the structure of $\Hcl$.

\begin{thm}\label{thm:Hall-Jordan}
\begin{enumerate}[nosep, label=(\arabic*)]
\item 
$\Hcl$ is commutative and cocommutative. 
\item 
As a $\bbC$-algebra we have 
$\Hcl \simeq \bbC[[I_{(1)}],[I_{(1^2)}],\ldots]$.
\item 
$\Delta([I_{(1^n)}])
=\sum_{r=0}^n q^{-r(n-r)}[I_{(1^r)}] \otimes [I_{(1^{n-r})}]$.
\end{enumerate}
\end{thm}

\begin{proof}
\begin{enumerate}[nosep]
\item
The commutativity is the consequence of certain duality on $\catA$.
For $N=(V,x) \in \Ob(\catA)$, 
the linear dual $V^*$ of $V$ and 
the transpose ${}^t x :V^* \to V^*$ of $x: V \to V$ 
yield $N^*:=(V^*, {}^t x) \in \Ob(\catA)$.
Note that we have $I_\lambda^* \simeq I_\lambda$.
Now  for 
$N \in 
\calG^{\lambda}_{\mu,\nu}
=\{N \subset I_{\lambda} \mid 
   N \simeq I_{\nu}, \, I_{\lambda}/N \simeq I_{\mu}\}$
we set 
$N^\perp := \{\xi \in I_{\lambda}^* \mid \xi(N)=0\}$.
Then since $N^\perp \simeq I_\mu$ and 
$I_\lambda^*/N^\perp \simeq I_\nu$,
the correspondence $N \mapsto N^\perp$ gives a bijection
\[
 \calG^{\lambda}_{\mu,\nu}
 \longsimto
 \{M \subset I_{\lambda}^* \simeq I_{\lambda} \mid 
   M \simeq I_\mu, \ I_{\lambda}^*/M \simeq I_{\nu}\}
 =\calG^{\lambda}_{\nu,\mu}.
\]
Thus we have 
$g^{\lambda}_{\mu,\nu}=g^{\lambda}_{\nu,\mu}$
for any partitions $\lambda,\mu, \nu$,
which implies the commutativity.

By the description \eqref{eq:Hall(A)} we find that 
the cocommutativity is a consequence of the commutativity.

\item
We follow \cite[Chap.I\!I \S2 (2.3)]{M} and \cite[\S2.2]{S}. 
For a later purpose we write down the proof.

Writing a partition $\lambda$ as 
$\lambda=(1^{l_1},2^{l_2}, \ldots, n^{l_n})$,
we set
\[
 X_{\lambda} := 
 [I_{(1^{l_n})}] * [I_{(1^{l_{n-1}+l_n})}] * \cdots * 
 [I_{(1^{l_1+ \cdots +l_n})}].
\]
It can be expanded as
\begin{equation}\label{eq:polynom:XI}
 X_\lambda = \sum_{\mu: \, \abs{\mu}=\abs{\lambda}}
 a_{\lambda \mu} [I_\mu], \quad 
 a_{\lambda \mu} \in \bbC.
\end{equation}
Assume $a_{\lambda \mu} \neq0$ and as an object of $\catA$
express $I_\mu$ as $I_\mu=(V,x)$.
Then by Lemma \ref{lem:filter},
there exists a filtration 
\[
 V^{\bullet}=(0=V^0 \subset V^1 \subset \cdots \subset V^n=V)
\]
such that for each $i=1,\ldots,n$ we have 
$\dim(V^i/V^{i-1})=l_i+ \cdots + l_n$ and $x(V^i) \subset V^{i-1}$.
In particular $\Ker x^i \supset V^i$, so we have 
\begin{equation}\label{eq:K>=V}
 \dim(\Ker x^i) \ge
 \dim V^i=l_1+2l_2+ \cdots +(i-1)i_{i-1}+i(l_i+\cdots+l_n).
\end{equation}

Now for a partition $\nu=(1^{n_1},2^{n_2}, \ldots)$ 
and $i \in \bbZ_{>0}$ we set
\begin{equation*}
 \sigma_i(\nu):=n_1+2n_2+ \cdots +(i-1)n_{i-1}+i(n_i+n_{i+1}+\cdots),
\end{equation*}
and define a partial order $\succeq$ on partitions by 
\begin{align}\label{eq:succeq}
 \alpha \succeq \beta \iff 
 \abs{\alpha} = \abs{\beta} 
 \text{ and for each $i \in \bbZ_{>0}$ we have }
 \sigma_i(\alpha) \le \sigma_i(\beta).
\end{align}
Since 
\begin{equation}\label{eq:sigma_i}
 \sigma_i(\nu) = \nu'_1+\cdots+\nu'_i,
\end{equation}
we have $\dim(\Ker x^i)=\sigma_i(\mu)$.
Then from  \eqref{eq:K>=V} we find that 
$a_{\lambda \mu} \neq 0$  $\Longrightarrow$ $\lambda \succeq \mu$.
Moreover if $\lambda = \mu$ then the filtration $V^\bullet$ 
is uniquely determined.
Therefore \eqref{eq:polynom:XI} has the form 
\begin{align*}
 X_{\lambda} = 
 [I_{\lambda}] + \sum_{\mu \prec \lambda} a_{\lambda \mu} [I_\mu],
 \quad 
 a_{\lambda \mu} \in \bbC.
\end{align*}
Then the matrix $A=(a_{\lambda \mu})_{\lambda,\mu: \tpar}$
is upper triangular and all the diagonal entries are $1$. 
Thus $A$ has its inverse matrix $A^{-1}$, and writing 
$A^{-1}=(a^{\lambda \mu})_{\lambda,\mu: \tpar}$ we have
\begin{equation}\label{eq:I-X}
 [I_{\lambda}]=\sum_{\mu \preceq \lambda}a^{\lambda \mu} X_{\lambda}.
\end{equation}
Thus $\Hcl$ is expanded by $X_\lambda$'s.
Since $X_\lambda$ is a product of $[I_{(1^n)}]$'s
and $\Hcl$ is commutative,
we find $\Hcl \simeq \bbC[[I_{(1)}],[I_{(1^2)}],\ldots]$.

\item
Explanation will be given in the last part of 
\S \ref{subsec:Hcl:pair}.
\end{enumerate}
\end{proof}

\begin{rmk}\label{rmk:dom}
By \eqref{eq:sigma_i}, 
the partial order \eqref{eq:succeq} is equivalent to 
the dominance order \cite[Chap.I \S1 p.7]{M}. 
Precisely speaking, we have 
\begin{align*}
\alpha \succeq \beta \iff
\alpha' \le \beta' \defiff
\abs{\alpha'}=\abs{\beta'}
\text{and for each $i\in\bbZ_{>0}$ we have }
\alpha'_1+\cdots+\alpha'_i \le \beta'_1+\cdots+\beta'_i.
\end{align*}
\end{rmk}

\subsection{Some examples of the structure constants and Pieri rule}

For $n,r \in \bbN$ with $n \ge r$, 
the $q$-binomial coefficient $\bnm{n}{r}_q$ is given by 
\[
 (x;q)_n := \prod_{i=1}^n (1-x q^{i-1}), \quad
 \begin{bmatrix} n \\ m \end{bmatrix}_q := 
 \frac{(q;q)_n}{(q;q)_r (q;q)_{n-r}}.
\]
For $n<r$, we set $\bnm{n}{r}:=0$.

Let us denote by $\GL(n)$ the general linear group of degree $n$,
and by $\Gr(n,r)$ the Grassmannian consisting of 
$r$-dimensional subspaces of the linear space $k^n$.
Then, over $k=\bbF_q$ we have
$\abs{\GL(n)} = (q^n-1)(q^n-q) \cdots (q^n -q^{n-1})$
and $\abs{\Gr(n,r)} = \bnm{n}{r}_q$.

\begin{lem}\label{lem:calG:1}
\begin{enumerate}[nosep, label=(\arabic*)]
\item
$\calG^{(1^n)}_{(1^{n-r}),(1^r)}  = \Gr(n,r)$.
In particular, for $k=\bbF_q$ we have 
$g^{(1^n)}_{(1^{n-r}),(1^r)} =\bnm{n}{r}_q$.

\item
$\calG^{(n)}_{(n-r),(r)}$ consists of one point.
In particular, for $k=\bbF_q$ we have 
$g^{(n)}_{(n-r),(r)}=1$.
\end{enumerate}
\end{lem}

\begin{proof}
\begin{enumerate}[nosep, label=(\arabic*)]
\item
Setting $M:=I_{(1)}^{\oplus n}$, 
we see that 
$\calG^{(1^n)}_{(1^{n-r}),(1^r)}=
 \{R \subset M \mid 
   R \simeq I_{(1)}^{\oplus r}, \, 
   M/R \simeq I_{(1)}^{\oplus n-r}\}$
consists of $r$-dimensional $k$-linear subspaces of $M$,
which is nothing but the Grassmannian $\Gr(n,r)$. 
\item 
The $r$-dimensional irreducible sub-representation of $I_{(n)}$
uniquely exists, which is $I_{(r)}$.
\end{enumerate}
\end{proof}

We turn to more general examples of $g^\lambda_{\mu,\nu}$.
Let us recall the notion of vertical strip.

\begin{dfn*}
Let $\lambda$ and $\mu$ be partitions.
\begin{enumerate}[nosep]
\item
We define $\lambda \supset \mu$ $\defiff$ 
$\lambda_i\ge\mu_i$ for any $i \in \bbZ_{>0}$.

\item
We call $\theta := \lambda-\mu$ a vertical strip 
if $0 \le \theta_i \le 1$ for any $i \in \bbZ_{>0}$.
\end{enumerate}
\end{dfn*}

We omit the proof of the following proposition.
See \cite[Chap.I\!I \S4 (4.4)]{M}.
In \cite[Example 2.4]{S} one may 
find a proof for a special case.

\begin{prp}\label{prp:vertical}
Consider the situation over an arbitrary field $k$.
Let $\beta$ and $\mu$ be partitions such that 
$\mu-\beta$ is a vertical strip.
Set $p:=\abs{\mu}-\abs{\beta} \in \bbN$, and 
assume $M,P\in \Ob(\catA)$ satisfy the following conditions.
\[
 M=I_\mu, \ P \subset M, \ P \simeq I_{(1^p)}, \ 
 M/P \simeq \beta.
\]
For a partition $\lambda$, 
we define the set $\calH_{\lambda}(P,M)$ by 
\[
 \calH_{\lambda}(P,M) := 
 \{L \subset P \mid M/L \simeq I_\lambda\}.
\]
Then,
$\calH_{\lambda}(P,M) \neq \emptyset$ $\iff$
$\beta \subset \lambda \subset \mu$
(so $\mu - \lambda$ is a vertical strip).
Moreover, if $\calH_{\lambda}(P,M) \neq \emptyset$, then 
\[
 \calH_{\lambda}(P,M) \simeq 
 \prod_{i=1}^m\Gr(\theta_i'+\varphi_i',\theta_i') \times 
 (\bbA^{\theta_i'} \times \bbA^{\sum_{j>i}\varphi_i'}),
\]
where $\bbA^d$ is the $d$-dimensional affine space over $k$
and $\varphi:=\lambda-\beta$, $\theta:=\mu-\lambda$.
\end{prp}

\begin{cor}[{Pieri formula for Hall-Littlewood polynomial}]
\label{cor:Pieri}
For partitions $\lambda,\mu$ and $p \in \bbN$, 
we have $\calG^{\nu}_{\lambda,(1^p)} \neq \emptyset \iff$
$\theta:=\nu-\mu$ is a vertical strip and $\abs{\theta}=p$.
Moreover if $\calG^{\nu}_{\lambda,(1^p)} \neq \emptyset$ and $k=\bbF_q$, 
then we have 
\[
  g^{\mu}_{\lambda,(1^p)} = 
  q^{n(\mu)-n(\lambda)-n(1^p)} \prod_{i\ge1} 
  \begin{bmatrix} \mu'_i - \mu'_{i+1} \\ 
                  \mu'_i - \lambda'_i \end{bmatrix}_{1/q}.
\]
\end{cor}

\begin{proof}
Apply Proposition \ref{prp:vertical} to the case
$\beta=\wt{\mu}:=\mu-(1^{\ell(\mu)})$.
Since in this case 
$M/P \simeq I_{\wt{\mu}}$ and $P=\Ker x \subset M=I_\mu$,
we have 
\[
  \calH_\lambda(P,M) 
= \{L \subset P     \mid M/L     \simeq I_\lambda\}
= \{L \subset I_\mu \mid I_\mu/L \simeq I_\lambda\}
= \calG^{\lambda}_{\mu,(1^p)}.
\]
Setting $m:=\mu_1$, we have 
\[
  g^{\lambda}_{\mu,(1^p)} = \abs{\calH_\lambda(\Ker x,I_{\mu})}
= \prod_{i=1}^m q^{\theta'_i (\varphi'_{i+1}+\cdots+\varphi'_m)}
  \begin{bmatrix} \theta'_i+\varphi'_i \\ \theta'_i \end{bmatrix}_q.
\]

Since
$\varphi:=\lambda-\wt{\mu}$, $\theta:=\mu-\lambda$
we have $\theta+\varphi=\mu-\wt{\mu}$.
Thus $\theta'_i+\varphi'_i=\mu'_i-\mu'_{i+1}$, 
$\theta'_i=\mu'_i-\lambda'_i$.
Then 
\[
 \begin{bmatrix}\theta'_i+\varphi'_i \\ \theta'_i\end{bmatrix}_q
=q^{\theta'_i \varphi'_i} 
 \begin{bmatrix}\theta'_i+\varphi'_i \\ \theta'_i\end{bmatrix}_{1/q}
=q^{\theta'_i \varphi'_i} 
 \begin{bmatrix}\mu'_i-\mu'_{i+1} \\ \mu'_i - \lambda'_i\end{bmatrix}_{1/q}.
\]
Therefore we have the $q$-binomial part of the statement.
The power of $q$ is given by 
\[
 \sum_i \bigl(
 \theta'_i \varphi'_i + \theta'_i(\varphi'_{i+1}+\cdots+\varphi'_m)\bigr)
=\sum_{i \le j}\theta'_i \varphi'_j.
\] 
Since $\theta$ and $\varphi$ are vertical strip, we have 
\[
 \sum_{i \le j} \theta'_i \varphi'_j =\sum_{i \le j} \varphi_j \theta_i
\]
Then since $\varphi_i=\lambda_i-\mu_i+1=1-\theta_i$ and $\theta_i^2=\theta_i$,
we have
\[
  \sum_{i \le j} \varphi_j \theta_i
 =\sum_{i \le j} (1-\theta_i)\theta_i 
 =\sum_{i \le j} (1-\theta_i)\theta_i.
\]
Now from $\sum_i \theta_i=p$ we have 
\begin{align*}
  \sum_{i \le j} (1-\theta_i)\theta_i 
&=\sum_{i < j} (1-\theta_i)\theta_i 
 =\sum_j (j-1)\theta_j - \bigl(\sum_i \theta_i\bigr)^2/2
 +\sum_i \theta_i^2/2
\\
&=n(\mu)-n(\lambda)-p(p-1)/2
 =n(\mu)-n(\lambda)-n(1^p).
\end{align*}
\end{proof}

\subsection{Hopf pairing}
\label{subsec:Hcl:pair}

Recall Green's Hopf pairing.
\[
 \pair{[I_\lambda]}{[I_\mu]} = \delta_{\lambda,\mu} a_\lambda^{-1},
 \quad
 a_\lambda := a_{I_\lambda} = \abs{\Aut(I_\lambda)}.
\] 
The formula of $a_\lambda$ is as follows.
See \cite[Chap.I \S1 (1.6)]{M} or \cite[Lemma 2.8]{S} for the proof.

\begin{lem}\label{lem:I_lambda:pair}
For a partition $\lambda$ we have 
\[
 \Aut(I_\lambda) \simeq \prod_{i\ge1} \GL(m_i) \times 
 \bigl(\bbA^{m_i} \times 
       \bbA^{\sum_{j<i} j m_j + (i-1) m_i + i \sum_{j>i}m_j}\bigr).
\]
Here we set $m_i:=m_i(\lambda) = \{j \mid \lambda_j = i\}$.
In particular, for $k=\bbF_q$ we have 
\[
 a_\lambda = q^{\abs{\lambda}+2n(\lambda)}
   \prod_{i\ge1}(q^{-1};q^{-1})_{m_i(\lambda)}
\]
with $n(\lambda) := \sum_{i \ge 1}(i-1)\lambda_i$.
Hence 
\[
 \pair{[I_\lambda]}{[I_\mu]} = 
 \frac{\delta_{\lambda,\mu}}
 {q^{\abs{\lambda}+2n(\lambda)} 
  \prod_{i\ge1}(q^{-1};q^{-1})_{m_i(\lambda)}}.
\]
\end{lem}

Let us close this subsection by checking Theorem \ref{thm:Hall-Jordan} (3).
We want to compute
\[
 \Delta([I_{(1^n)}]) = 
 \sum_{r=0}^n  a_{(1^n)}^{-1} a_{(1^{n-r})} a_{(1^r)} 
 g^{(1^n)}_{(1^{n-r}),(1^r)} \cdot [I_{(1^{n-r})}] \otimes [I_{(1^r)}].
\]
By Lemma \ref{lem:I_lambda:pair} we have 
$a_{(1^n)}=(-1)^n q^{\binom{n}{2}} (q;q)_n$.
By Lemma \ref{lem:calG:1} (1) we have 
$g^{(1^n)}_{(1^{n-r}) (1^r)}=\bnm{n}{r}_q$.
Then
\[
 a_{(1^n)}^{-1} a_{(1^{n-r})} a_{(1^r)} 
 g^{(1^n)}_{(1^{n-r}),(1^r)} 
=q^{-\binom{n}{2}} q^{\binom{n-r}{2}} q^{\binom{r}{2}}
=q^{-r(n-r)},  
\]
which deduces 
$\Delta([I_{(1^n)}])
=\sum_{r=0}^n q^{-r(n-r)}[I_{(1^r)}] \otimes [I_{(1^{n-r})}]$.

\subsection{A computation of antipode}

The following formula for the antipode $S$ on $\Hcl$ 
is stated in \cite[\S2.3]{S} without a proof.

\begin{prp}\label{prp:S}
\[
 S([I_{(1^n)}]) =(-1)^n q^{-\binom{n}{2}}
 \sum_{|\lambda|=n} [I_{\lambda}].
\]
\end{prp}

Let us give a proof of this formula.
We will use the well-known 

\begin{lem}[{terminating $q$-binomial theorem}]\label{lem:q-binom}
We have
\[
 (x;q)_n = \sum_{k=0}^n (-x)^k q^{\binom{k}{2}}
           \begin{bmatrix} n \\ k \end{bmatrix}_q.
\]
\end{lem}

\begin{proof}
We will show the statement by induction on $n$ 
using the defining property of antipode: 
\begin{equation}\label{eq:antipode}
 m \circ (\id \otimes S)\circ \Delta([I_{(1^n)}]) = 
 i \circ \epsilon([I_{(1^n)}])=0.
\end{equation}
In the case $n=1$,
\eqref{eq:antipode} implies $S([I_{(1)}])+I_{(1)}=0$,
so that the statement holds.

For a general $n$, from 
$\Delta([I_{(1^n)}])
=\sum_{u=0}^n q^{-u(n-u)}[I_{(1^u)}] \otimes [I_{(1^{n-u})}]$
in Theorem \ref{thm:Hall-Jordan} (3) we see
\begin{align*}
 m \circ (\id \otimes S)\circ \Delta([I_{(1^n)}]) 
&=m \circ (\id \otimes S)
 \Bigl(\sum_{u=0}^n q^{-u(n-u)}[I_{(1^u)}] \otimes [I_{(1^{n-u})}]\Bigr)
\\
&= \sum_{u=0}^n q^{-u(n-u)}[I_{(1^u)}] * S([I_{(1^{n-u})}]).
\end{align*}
Thus
\begin{align}\label{eq:S:1}
 S([I_{(1^n)}])=
 -\sum_{u=1}^n q^{-u(n-u)}[I_{(1^u)}] * S([I_{(1^{n-u})}]).
\end{align}
By the Pieri formula in Corollary \ref{cor:Pieri}
and the commutativity of the product $*$, we have
\[
 [I_{(1^u)}]*[I_{\nu}] 
 = \sum_\lambda g^{\lambda}_{(1^u),\nu}[I_\lambda],
 \quad
 g^{\lambda}_{(1^u),\nu} = 
 q^{n(\lambda)-n(\nu)-n(1^u)} \prod_{i\ge1} 
 \begin{bmatrix} \lambda'_i - \lambda'_{i+1} \\ 
                 \lambda'_i - \nu'_i \end{bmatrix}_{1/q}.
\]
Then by the induction hypothesis we proceed with \eqref{eq:S:1} as 
\begin{align*}
-&\sum_{u=1}^n q^{-u(n-u)}[I_{(1^u)}] * S([I_{(1^{n-u})}])
\\
&=-\sum_{u=1}^n q^{-u(n-u)} \sum_{|\nu|=n-u}
   (-1)^{n-u}q^{-\binom{n-u}{2}}[I_{(1^u)}] * [I_{\nu}]
\\
&=-\sum_{u=1}^n q^{-u(n-u)}\sum_{|\nu|=n-u}
   (-1)^{n-u}q^{-\binom{n-u}{2}}
   \sum_{|\lambda|=n}q^{n(\lambda)-n(\nu)-\binom{u}{2}}
   \begin{bmatrix}
    \lambda'_i-\lambda'_{i+1} \\ \lambda'_i-\nu'_i
   \end{bmatrix}_{1/q} 
   [I_{\lambda}]
\\
&=\sum_{u=1}^n(-1)^{n-u+1}\sum_{|\nu|=n-u}
  q^{-n(\nu)-\binom{n}{2}}\sum_{|\lambda|=n}q^{n(\lambda)}
 \begin{bmatrix}
  \lambda'_i-\lambda'_{i+1}\\ \lambda'_i-\nu'_i
 \end{bmatrix}_{1/q} [I_{\lambda}].
\end{align*}
Therefore it is enough to show
\[
 \sum_{u=1}^n(-1)^{n-u+1}\sum_{|\nu|=n-u}
 q^{n(\lambda)-n(\nu)-\binom{n}{2}} \prod_{i \geq 1}
 \begin{bmatrix}
  \lambda'_i - \lambda'_{i+1} \\ \lambda'_i - \nu'_i
 \end{bmatrix}_{1/q}
 =(-1)^n q^{-\binom{n}{2}}.
\]
This is equivalent to 
\begin{align}\label{eq:S:final}
 \sum_{u=0}^n (-1)^u \sum_{|\nu|=n-u} q^{n(\lambda)-n(\nu)}
 \prod_{i \ge 1} 
 \begin{bmatrix}
 \lambda'_i - \lambda'_{i+1} \\ \lambda'_i - \nu'_i
 \end{bmatrix}_{1/q} =0.
\end{align}


In the case $\lambda=(r^s)$, 
the summation is over $\nu=(r^{s-u},(r-1)^u)$ with $0 \le u \le s$,
and then we have 
$n(\lambda)-n(\nu) = \sum_{i\ge1}(i-1)(\lambda_i-\nu_i)
=\sum_{i=s-u+1}^s(i-1)=s u - \binom{u+1}{2}$.
Thus \eqref{eq:S:final} is equivalent to 
\[
 \sum_{u=0}^s (-1)^u q^{s u - \binom{u+1}{2}}
 \begin{bmatrix} s \\ u \end{bmatrix}_{1/q}=0
 \iff
 \sum_{u=0}^s (-1)^u q^{\binom{u}{2}}
 \begin{bmatrix} s \\ u \end{bmatrix}_q=0.
\]
Since the $q$-binomial formula yields
$\sum_{u=0}^s (-1)^u q^{\binom{u}{2}} \bnm{s}{u}_q = (1;q)_s = 0$,
we are done.

In the general case 
$\lambda=(r_1^{s_1},r_2^{s_2},\ldots,r_l^{s_l})$,
the summation is over 
\[
 \nu=\left(r_1^{s_1-u_1},(r_1-1)^{u_1}, r_2^{s_2-u_2},(r_2-1)^{u_2},
   \ldots, r_l^{s_l-u_l},(r_l-1)^{u_l}\right)
\]
with $0 \le u_j \le s_j$ ($j=1,\ldots,l$). 
We then have 
\begin{equation}\label{eq:n(lambda)-n(mu)}
\begin{split}
  n(\lambda)-n(\nu)
&=\sum_{i\ge1}(i-1)(\lambda_i-\nu_i)
 =\sum_{j=1}^l \sum_{i=s_j-u_j+1}^{s_j}(s_1+\cdots+s_{j-1}+i-1)
\\
&=\sum_{j=1}^l 
  \Bigl(u_j(s_1+\cdots+s_{j-1}) + s_j u_j-\binom{u_j+1}{2}\Bigr).
\end{split}
\end{equation}
So the left hand side of \eqref{eq:S:final} is equal to
\begin{align*}
 &\sum_{(u_1,\ldots,u_l)} (-1)^{u_1+\ldots+u_l}
  \prod_{j=1}^l 
  q^{u_j(s_1+\cdots+s_{j-1})+s_j u_j-\binom{u_j+1}{2}}
  \begin{bmatrix} s_j \\ u_j \end{bmatrix}_{1/q}
\\
=&\sum_{(u_1,\ldots,u_l)} (-1)^{u_1+\ldots+u_l}
  \prod_{j=1}^l 
  q^{u_j(s_1+\cdots+s_{j-1})+\binom{u_j}{2}}
\begin{bmatrix} s_j \\ u_j \end{bmatrix}_q
\\
=&\Bigl(\sum_{u_1=0}^{s_1} (-1)^{u_1} q^{\binom{u_1}{2}} 
        \begin{bmatrix} s_1 \\ u_1 \end{bmatrix}_q \Bigr)
  \Bigl(\sum_{u_2=0}^{s_2} (-q^{s_1})^{u_2} q^{\binom{u_2}{2}} 
        \begin{bmatrix} s_2 \\ u_2 \end{bmatrix}_q \Bigr)
  \cdots
  \Bigl(\sum_{u_l=0}^{s_l} 
        (-q^{s_1+\cdots+s_{l-1}})^{u_l} q^{\binom{u_l}{2}} 
        \begin{bmatrix} s_l \\ u_l \end{bmatrix}_q \Bigr).
\end{align*}
Then the summation over $u_1$ is zero 
by the $q$-binomial formula,
so \eqref{eq:S:final} holds.
\end{proof}

\section{Hall-Littlewood symmetric function via classical Hall algebra}
\label{sec:HL}

We continue to study the classical Hall algebra $\Hcl$,
i.e., the Ringel-Hall algebra of the category 
$\catA=\Rep^{\nil}_{\bbF_q} Q$ of the Jordan quiver.
In this section we analyze symmetric functions via $\Hcl$.

\subsection{Hopf algebra of symmetric functions}
\label{ss:Lambda}

We follow \cite[Chap.I]{M} for the notations of symmetric functions.
In particular, the ring of symmetric function with variables 
$x=(x_1,x_2,\ldots)$ is denoted by $\Lambda(x)$,
which is defined as the limit of the projective system
of the rings of symmetric polynomials 
$\Lambda_n(x) := \bbC[x_1,\ldots,x_n]^{S_n}$
together with the projection maps 
\[
 \Lambda_{n+1}(x) \longsurj \Lambda_n(x),
 \quad
 x_i \longmapsto 
 \begin{cases}x_i & (1 \le i \le n) \\ 0 & (i=n+1)\end{cases}.
\]
We sometimes denote $\Lambda_n:=\Lambda_n(x)$ and $\Lambda:=\Lambda(x)$ 
if no confusion may occur.
The grading given by the polynomial degree is denoted as 
$\Lambda_n(x)=\oplus_{d \in \bbN} \Lambda_n^d(x)$ and  
$\Lambda(x)=\oplus_{d \in \bbN} \Lambda^d(x)$.

The elementary symmetric polynomial of $n$ variables is denoted by 
\[
 e_r(x_1,\ldots,x_n) = 
 \sum_{1 \le i_1 < \cdots < i_r \le n}
 x_{i_1} \cdots x_{i_r} \ 
 \in \Lambda^r_n(x).
\]
The family $\{e_r(x_1,\ldots,x_n) \mid n \in \bbN \}$ 
determines the elementary symmetric function 
$e_r(x) \in \Lambda^r(x)$.
For a partition $\lambda=(\lambda_1,\lambda_2,\ldots)$,
we define $e_\lambda(x) \in \Lambda^{|\lambda|}(x)$ by 
\[
 e_\lambda(x) := e_{\lambda_1}(x)e_{\lambda_2}(x)\cdots. 
\]
Then $\Lambda(x)$ has a basis $\{e_\lambda(x) \mid \tpar\}$
and as a ring we have $\Lambda(x) = \bbC[e_1(x),e_2(x),\ldots]$.

%
%

Let us recall the well-known fact due to Zelevinsky \cite{Z}.
See also \cite[Chap.I, I\!I\!I]{M} for an account.

\begin{fct*}
$\Lambda(x)$ is a Hopf algebra with the coproduct $\Delta$ given by 
\[
 \Delta(e_n(x)) := \sum_{r=0}^{n} e_r(x) \otimes e_{n-r}(y)  
 \in \Lambda(x) \otimes \Lambda(y).
\]
$\Lambda(x)$ also has a Hopf pairing.
Using $m_i(\lambda) = \{j \mid \lambda_j = i\}$,
the pairing is  given by 
\[
 \pair{e_\lambda(x)}{e_\mu(x)} :=
 \delta_{\lambda,\mu}\prod_{i\ge1} (q;q)_{m_i(\lambda)}^{-1}.
\]
\end{fct*}

Comparing this fact and Theorem \ref{thm:Hall-Jordan},
we immediately have 

\begin{thm}\label{thm:Hall=Lambda}
There is a Hopf algebra isomorphism
\[
 \psi: \Hcl \longsimto \Lambda(x),
 \quad [I_{(1^n)}] \longmapsto q^{-\binom{n}{2}} e_n(x).
\]
For a homogeneous element 
$a \in \Hcl$ with degree $n \in \bbZ = K_0(\catA)$,
we have
\[
 \pair{\psi(a)}{\psi(a)} = \pair{a}{a} \cdot q^n.
\]
\end{thm}

\subsection{Hall-Littlewood element}

Since $\Hcl$ has an orthogonal basis 
$\{[I_\lambda] \mid \lambda: \tpar\}$ 
with respect to the Hopf pairing $\pair{\cdot}{\cdot}$,
it is natural to name this basis.

\begin{dfn}\label{dfn:P_lambda}
For a partition $\lambda$, 
define $P_\lambda \in \Hcl$ by 
\[
 P_{\lambda} := q^{n(\lambda)} [I_{\lambda}]
\]
with $n(\lambda)=\sum_{i\ge1}(i-1)\lambda_i$.
We call $P_\lambda$ the Hall-Littlewood element.
For $n \in \bbN$, define $e_n \in \Hcl$ by
\[
 e_n := P_{(1^n)} = q^{\binom{n}{2}}[I_{(1^n)}],
\]
and for a partition $\lambda=(\lambda_1,\lambda_2,\ldots)$
we set $e_\lambda := e_{\lambda_1}*e_{\lambda_2}*\cdots$.
\end{dfn}

\begin{rmk*}
Under the isomorphism $\psi$ in Theorem \ref{thm:Hall=Lambda},
$P_\lambda$ corresponds to the Hall-Littlewood symmetric function 
$P_\lambda(x;q^{-1}) \in \Lambda(x)$.
and $e_n$ corresponds to the elementary symmetric function $e_n(x)$.
See \cite[Chap.I\!I\!I]{M} for the account of $P_\lambda(x;t)$.
\end{rmk*}

\subsection{Primitive elements of classical Hall algebra}

An element $p$ of a coalgebra is called primitive if
$\Delta(p)=p\otimes 1 + 1 \otimes p$.
We determine the primitive elements of 
the classical Hall algebra $\Hcl$.

\begin{dfn}\label{dfn:p_n}
Set $p_0 := 1 \in \Hcl$,
and for $n \in \bbZ_{>0}$ define $p_n \in \Hcl$ by 
\[
 p_n := \sum_{|\lambda|=n} 
 (q;q)_{\ell(\lambda)-1}\cdot [I_\lambda]. 
\]
\end{dfn}

Recall that $\Hcl$ has a grading by $K_0(\catA) = \bbZ$.

\begin{thm}\label{thm:prim}
For $n \in \bbZ_{>0}$,  
$p_n$ is a primitive element, i.e.,
$\Delta(p_n) = p_n \otimes 1 + 1 \otimes p_n$.
Conversely, 
homogeneous primitive elements of $\Hcl$
are equal to $p_n$ up to scalar multiplication.
\end{thm}

The proof of the following lemma will be postponed.

\begin{lem}\label{lem:p-e}
Let $\exp(x) := \sum_{n\ge0} (x* \cdots * x)/n!$
for $x \in \Hcl$.
Then the following equality holds in $\Hcl[[z]]$.
\[
 \sum_{n \geq 0}e_n (-z)^n 
 = \exp\Bigl(-\sum_{n \geq 1}\frac{1}{n}p_n z^n\Bigr).
\]
\end{lem}

\begin{proof}[\textbf{\emph{Proof of 
the first half of Theorem \ref{thm:prim}}}]
We will show 
\[
 \Delta(p_n)=p_n \otimes 1 + 1 \otimes p_n.
\]
By Theorem \ref{thm:Hall-Jordan} (3), we have 
\[
 \Delta(e_n)=\sum_{m = 0}^n e_m \otimes e_{n-m}.
\]
Setting $E(z):=\sum_{n \geq 0}e_n (-z)^n$, 
one can restate this equality as 
\[
 \Delta(E(z))=E(z) \otimes E(z).
\]
Using the equality 
$E(z)=\exp(-\sum_{n \geq 1}\frac{1}{n}p_n z^n)$
in Lemma \ref{lem:p-e} and 
$x \otimes x = (x \otimes 1)*(1\otimes x)$,
we have 
\begin{align*}
 \exp\Bigl(-\sum_{n \geq 1}\frac{1}{n}\Delta (p_n)z^n\Bigr)
=\Bigl(\exp\Bigl(-\sum_{n \geq 1}\frac{1}{n}p_n z^n\Bigr) \otimes 1\Bigr)*
 \Bigl(1 \otimes \exp\Bigl(-\sum_{n \geq 1}\frac{1}{n}p_n z^n\Bigr)\Bigr).
\end{align*}
Taking the logarithm, we have 
\begin{align*}
-\sum_{n \geq 1}\frac{1}{n} \Delta(p_n)z^n
=\Bigl(-\sum_{n \geq 1}\frac{1}{n}p_n z^n \otimes 1\Bigr)
+\Bigl(1 \otimes -\sum_{n \geq 1}\frac{1}{n}p_n z^n\Bigr).
\end{align*} 
Comparing the coefficients of $z^n$ we have the statement.
\end{proof}

\begin{proof}[\textbf{\emph{Proof of Lemma \ref{lem:p-e}}}]
One can see by induction that the conclusion is equivalent to 
the equality
\begin{equation}\label{eq:p-e}
 \sum_{r=0}^{n-1}(-1)^r p_{n-r} * e_r = (-1)^{n-1}n e_n
\end{equation}
for each $n \in \bbZ_{>0}$.
So let us show \eqref{eq:p-e}.

By the definitions of $p_n$, $e_n$ and $P_\lambda$, we have 
\begin{align*}
\text{(LHS of \eqref{eq:p-e})} 
=\sum_{r=0}^{n-1}(-1)^r \sum_{|\mu|=n-r}
  q^{-n(\mu)} (q;q)_{\ell(\mu)-1} P_{\mu} * P_{(1^r)},
\quad
\text{(RHS of \eqref{eq:p-e})} 
=(-1)^{n-1} n P_{(1^n)}.
\end{align*}
By Corollary \ref{cor:Pieri} we have 
\[
 P_{\mu} * P_{(1^r)} = 
 \sum_{\substack{\abs{\lambda}=n \\ \lambda-\mu: \text{vertical strip}}} 
 q^{n(\mu)+n(1^r)-n(\lambda)} g^{\lambda}_{\mu,(1^r)} P_\lambda,
 \quad
 g^{\lambda}_{\mu,(1^r)} = 
 q^{n(\lambda)-n(\mu)-n(1^r)} \prod_{i\ge1} 
 \begin{bmatrix} \lambda'_i - \lambda'_{i+1} \\ 
                 \lambda'_i - \mu'_i \end{bmatrix}_{1/q}.
\]
Since $\{P_\lambda \mid \lambda: \tpar\}$
is a basis of $\Hcl$, we find that 
\eqref{eq:p-e} is equivalent to the equalities 
\begin{align}
\label{eq:p-e:2}
 \sum_{i=0}^{n-1} (-1)^i 
 \sum_{\substack{\abs{\mu}=n-i \\  \\ \lambda-\mu: \text{vertical strip}}} 
 q^{-n(\mu)} (q;q)_{\ell(\mu)-1} 
 \begin{bmatrix} \lambda'_i - \lambda'_{i+1} \\ 
                 \lambda'_i - \mu'_i \end{bmatrix}_{1/q}
 = \delta_{\lambda,(1^n)}(-1)^{n-1} n
\end{align}
for any partition $\lambda$ with $\abs{\lambda}=n$.
So let us show \eqref{eq:p-e:2}.

First consider the case $\lambda=(1^n)$.
Then the summation in the left hand side of \eqref{eq:p-e:2}
is over $\mu=(1^{n-u})$ with $0 \le u \le n-1$, so we have 
\begin{align*}
\text{(LHS of \eqref{eq:p-e:2})}
=\sum_{u=0}^{n-1}(-1)^u q^{-\binom{n-u}{2}} (q;q)_{n-u-1}
 \begin{bmatrix}n\\ i\end{bmatrix}_{1/q}
=(-1)^{n-1} \sum_{u=0}^{n-1} (1/q;1/q)_{n-u-1}
 \begin{bmatrix}n \\ u\end{bmatrix}_{1/q}.
\end{align*}
Thus, denoting $t:=1/q$ we have 
\[
 \left(\eqref{eq:p-e:2} \text{ for }\lambda=(1^n)\right) 
 \iff 
 \sum_{u=1}^n \begin{bmatrix}n \\ u\end{bmatrix}_t (t;t)_{u-1}=n.
\]
One can show this equality by induction on $n$ 
using the recursive formula
$\bnm{n}{u}_t=t^u\bnm{n-1}{u}_t+\bnm{n-1}{u-1}_t$.

Next consider the case $\lambda=(r^s)$ with $r>1$.
The summation 
is over $\mu=(r^{s-u}, (r-1)^{u})$ with $0 \le u \le s$.
Since $n(\mu)=(r-1)\binom{s}{2}+\binom{u}{2}$
and $\bnm{s}{u}_{1/q}=q^{-u(s-u)}\bnm{s}{u}_q$, 
we have
\[
 \left(\text{LHS of \eqref{eq:p-e:2}}\right) 
=q^{-(r-1)\binom{s}{2}} (q;q)_{s-1}
 \sum_{u=0}^s (-1)^{u} q^{-\binom{u}{2}} 
 \begin{bmatrix}s \\ u\end{bmatrix}_{1/q}
=q^{-(r-1)\binom{s}{2}} (q;q)_{s-1}
 \sum_{u=0}^s (-u^{1-s})^u q^{\binom{u}{2}} 
 \begin{bmatrix}s \\ u\end{bmatrix}_q.
\]
The $q$-binomial theorem (Lemma \ref{lem:q-binom}) yields
$\sum_{u=0}^s (-u^{1-s})^u q^{\binom{u}{2}} \bnm{s}{u}_q
=(u^{1-s};u)_s=0$,
so the case $\lambda=(r^s)$ is proved.

In the general case $\lambda =(r_1^{s_1},\ldots ,r_l^{s_l})$ 
with $l>1$,
the summation is over 
\[
 \mu=\left(r_1^{s_1-u_1},(r_1-1)^{u_1}, r_2^{s_2-u_2},(r_2-1)^{u_2},
   \ldots, r_l^{s_l-u_l},(r_l-1)^{u_l}\right)
\]
with $0 \le u_j \le s_j$.
Assume that $r_l>1$.
Then $\ell(\mu)=\ell(\lambda)$ so we have 
\begin{align*}
 \left(\text{LHS of \eqref{eq:p-e:2}}\right) 
=(q;q)_{\ell(\lambda)-1}
 \sum_{(u_1, \ldots ,u_l)}
 (-1)^{u_1+\cdots +u_l}q^{-n(\mu)}
 \begin{bmatrix}s_1\\ u_1\end{bmatrix}_{1/q} \cdots 
 \begin{bmatrix}s_l\\ u_l\end{bmatrix}_{1/q}.
\end{align*}
By \eqref{eq:n(lambda)-n(mu)} in the proof of 
Proposition \ref{prp:S} we have 
\begin{align*}
n(\mu)=n(\lambda)-
\sum_{j=1}^l \bigl(u_j(s_1+\cdots+s_{j-1})
             +u_j s_j-\binom{u_j+1}{2}\bigr).
\end{align*}
Then using $\bnm{s}{u}_{1/q}=q^{u(u-s)}\bnm{s}{u}_q$ we have 
\begin{align*}
\left(\text{LHS of \eqref{eq:p-e:2}}\right)
=&(q;q)_{\ell(\lambda)-1} q^{-n(\lambda)} 
  \Bigl(\sum_{u_1=0}^{s_1}
        (-1)^{u_1} q^{\binom{u_1}{2}}
        \begin{bmatrix}s_1 \\ u_1\end{bmatrix}_q\Bigr)
  \cdot 
  \Bigl(\sum_{u_2=0}^{s_2} 
        (-q^{s_1})^{u_2} q^{\binom{u_2}{2}}
        \begin{bmatrix}s_2 \\ u_2\end{bmatrix}_q\Bigr)\cdot
 \\
 &\cdots \cdot
   \Bigl(\sum_{u_l=0}^{s_l} 
         (-q^{s_1+\cdots +s_{l-1}})^{u_l} q^{\binom{u_l}{2}}
         \begin{bmatrix}s_l \\ u_l\end{bmatrix}_q\Bigr).
\end{align*}
The $q$-binomial theorem (Lemma \ref{lem:q-binom}) gives 
$\sum_{u_1=0}^{s_1}
 (-1)^{u_1} q^{\binom{u_1}{2}}
 \bnm{s_1}{u_1}_q=0$,
so \eqref{eq:p-e:2} is shown under the assumption $r_l>1$.

The final case is 
$\lambda =(r_1^{s_1},\ldots ,r_l^{s_l})$
with $l>1$ and $r_l=1$.
The summation is over 
\[
 \mu=\bigl(r_1^{s_1-u_1},(r_1-1)^{u_1}, \ldots,
           r_{l-1}^{s_{l-1}-u_{l-1}},(r_{l-1}-1)^{u_{l-1}},
           1^{s_l-u_l}\bigr)
\]
with $0 \le u_j \le s_j$.
Then $\ell(\mu)=s_1+\cdots+s_l-u_l=\ell(\wt{\lambda})+s_l-u_l$
with $\wt{\lambda}:=\lambda-(1^{\ell(\lambda)})$, so we have 
\begin{align*}
 \left(\text{LHS of \eqref{eq:p-e:2}}\right) 
=&(q;q)_{\ell(\wt{\lambda}} q^{-n(\lambda)} 
  \Bigl(\sum_{u_1=0}^{s_1}
        (-1)^{u_1} q^{\binom{u_1}{2}}
        \begin{bmatrix}s_1 \\ u_1\end{bmatrix}_q\Bigr)
  \cdot 
  \Bigl(\sum_{u_2=0}^{s_2} 
        (-q^{s_1})^{u_2} q^{\binom{u_2}{2}}
        \begin{bmatrix}s_2 \\ u_2\end{bmatrix}_q\Bigr)\cdot
 \\
 &\cdots \cdot
   \Bigl(\sum_{u_l=0}^{s_l} 
         (q;q)_{s_l-u_l}
         (-q^{s_1+\cdots +s_{l-1}})^{u_l} q^{\binom{u_l}{2}}
         \begin{bmatrix}s_l \\ u_l\end{bmatrix}_q\Bigr).
\end{align*}
Thus by the same reason as in the case $r_l>1$,
the first summation vanishes and the proof is completed.

\end{proof}

Next we turn to the proof of the latter part of Theorem \ref{thm:prim}.
Let us note 

\begin{lem}\label{lem:p:basis}
For a partition $\lambda=(\lambda_1, \lambda_2, \ldots)$,
we define $p_\lambda \in \Hcl$ by 
$p_{\lambda} := p_{\lambda_1} * p_{\lambda_2} * \cdots$.
Then 
$\{p_{\lambda} \mid \lambda: \tpar\}$
is a basis of $\Hcl$.
\end{lem}

\begin{proof}
For a partition $\lambda=(\lambda_1,\lambda_2,\ldots)$,
we set $e_\lambda := e_{\lambda_1} * e_{\lambda_2} * \cdots$.
By Theorem \ref{thm:Hall-Jordan} (2),
$\{e_\lambda \mid \lambda: \tpar\}$ is a basis of $\Hcl$.
Using \eqref{eq:p-e} repeatedly, we find that 
any $e_\lambda$ can be written by a linear combination of $p_\mu$'s.
Considering the graded dimension, we obtain the conclusion.
\end{proof}

\begin{proof}[\textbf{\emph{Proof of the latter half of Theorem \ref{thm:prim}}}]
Assume that $x \in \Hcl$ is primitive and homogeneous of degree $n$.
By Lemma \ref{lem:p:basis} we an write 
$x=\sum_{\abs{\lambda}=n} c_{\lambda}p_{\lambda}$, $c_\lambda \in \bbC$.
Then using some $y_\lambda \in \Hcl$ we have 
\[
\Delta(x)
 =\sum_{\abs{\lambda}=n} c_{\lambda}\Delta(p_{\lambda})
 =\sum_{\abs{\lambda}=n} c_{\lambda}
  \prod_{i=1}^{\ell(\lambda)}
  \left(p_{\lambda_i}\otimes 1+1 \otimes p_{\lambda_i}\right)
 =\sum_{\abs{\lambda}=n} c_{\lambda}
  \left(p_{\lambda}\otimes 1+y_\lambda+1\otimes p_{\lambda}\right).
\]
By this calculation we find 
$\Delta(x) =  x\otimes 1+ 1 \otimes x$ $\iff$
$y_\lambda = 0$ for any $\lambda$.
Since 
$y_\lambda
=\sum_{1<\ell(\nu)+\ell(\mu)<\ell(\lambda)} p_{\nu}\otimes p_{\mu}$,
we have $y_\lambda=0$ $\iff$ $\ell(\lambda)=1$.
Thus $\Delta(x)=x \otimes 1 + 1 \otimes x$ implies 
$x=c_{(n)} p_n$.
\end{proof}

\subsection{The Hopf pairing of $p_n$'s}

\begin{thm}\label{thm:p_n:pair}
For $m,n \in \bbN$ we have 
\[
 \pair{p_m}{p_n} = \delta_{m,n}\frac{n}{q^n-1}.
\]
\end{thm}

Precisely speaking, 
for $m,n \in \bbZ_{>0}$ we have the above formula,
for $m n = 0$ and $(m,n)\neq (0,0)$ we have $\pair{p_m}{p_n}=0$,
and for $m=n=0$ we have $\pair{p_0}{p_0}=1$.

\begin{rmk*}
Under the isomorphism 
$\psi: \Hcl \simto \Lambda$
in Theorem \ref{thm:Hall=Lambda} we have 
$\pair{\psi(p_m)}{\psi(p_n)} = \delta_{m,n} n/(1-q^{-n})$,
which is nothing but the inner product on $\Lambda$ 
for Hall-Littlewood symmetric functions \cite[Chap.V\!I \S1]{M} 
with the replacement $t=q^{-1}$.
\end{rmk*}

\begin{proof}
If the $K_0(\catA)=\bbZ$-gradings are different, 
then the pairing is $0$.
So it is enough to check $\pair{p_n}{p_n} = n/(q^n-1)$.
By Definition \ref{dfn:p_n} of $p_n$ and 
the pairing of $[I_\lambda]$ (Lemma \ref{lem:I_lambda:pair}),
we have 
\begin{equation}\label{eq:p_n^2}
 \pair{p_n}{p_n}
 = \sum_{\abs{\lambda}=n}
   \frac{(q;q)_{\ell(\lambda)-1}^2}
        {q^{\abs{\lambda}+2n(\lambda)} 
         \prod_{i\ge1}(q^{-1};q^{-1})_{m_i(\lambda)}}
 = q^{-n} \sum_{\abs{\lambda}=n} 
   q^{-2n(\lambda)+2\binom{\ell(\lambda)}{2}}
   \frac{(q^{-1};q^{-1})_{\ell(\lambda)-1}^2}
        {\prod_{i\ge1}(q^{-1};q^{-1})_{m_i(\lambda)}}.
\end{equation}
Now we demand that the statement can be deduced from 
the following equality for indeterminates $t$ and $u$: 
\begin{align}\label{eq:p_n:pair:i}
 \sum_{|\lambda|=n} 
 t^{2n(\lambda)-2\binom{\ell(\lambda)}{2}}
 \frac{ (t;t)_{\ell(\lambda)-1} (u^{-1};t)_{\ell(\lambda)}}
      {\prod_{i\ge1}(t;t)_{m_i(\lambda)}}
 u^{\ell(\lambda)}
=\frac{u^n-1}{1-t^n}.
\end{align}
Actually, dividing both sides by $u-1$, taking the limit $u \to 1$,
and finally setting $t=q^{-1}$, we have 
\begin{align*}
 \frac{\text{(LHS of \eqref{eq:p_n:pair:i})}}{u-1} 
 \to \sum_{|\lambda|=n} 
 q^{-n(\lambda)+2\binom{\ell(\lambda)}{2}}
 \frac{(q^{-1};q^{-1})_{\ell(\lambda)-1}(q^{-1};q^{-1})_{\ell(\lambda)-1}}
      {\prod_i(q^{-1};q^{-1})_{m_i(\lambda)}},
 \quad
 \frac{\text{(RHS of \eqref{eq:p_n:pair:i})}}{u-1} 
 \to \frac{n}{1-q^{-n}},
\end{align*}
so we find that $\eqref{eq:p_n^2}=n/(q^n-1)$.

Let us introduce the $t$-multinomial coefficient. 
For $l,m_1,\ldots,m_n$ with $n \ge 2$, we set 
\[
 \begin{bmatrix} l \\ m_1, m_2, \ldots, m_n \end{bmatrix}_t
 :=\frac{(t;t)_l}{(t;t)_{m_1}(t;t)_{m_2} \cdots (t;t)_{m_n}}.
\]
For example, if $n=2$ and $l=m_1+m_2$, then 
$\bnm{l}{m_1, m_2}_t=\bnm{l}{m_1}_t=\bnm{l}{m_2}_t$.
Noticing that 
$\ell(\lambda)=m_1(\lambda)+\cdots+m_n(\lambda)$,
we can rewrite \eqref{eq:p_n:pair:i} as 
\begin{align}
\nonumber
 \left(\text{LHS of \eqref{eq:p_n:pair:i}}\right)
&=\sum_{\abs{\lambda}=n} 
   t^{2n(\lambda)-\ell(\lambda)(\ell(\lambda)-1)}
   \begin{bmatrix}
    l(\lambda) \\ m_1(\lambda), \ldots , m_n(\lambda)
   \end{bmatrix}_t 
   \frac{(u^{-1};t)_{l(\lambda)}}{1-t^{\ell(\lambda)}}
   u^{\ell(\lambda)}
\\
\label{eq:p_n:pair:i'}
&=\sum_{\abs{\lambda}=n}
   t^{-\ell(\lambda)(\ell(\lambda)-1)} u^{\ell(\lambda)}
   \frac{(u^{-1};t)_{\ell(\lambda)}}{1-t^{\ell(\lambda)}}
   \cdot
   t^{2n(\lambda)}
   \begin{bmatrix}
    \ell(\lambda) \\ m_1(\lambda), \ldots , m_n(\lambda)
   \end{bmatrix}_t. 
\end{align}
Note that the factors before $\cdot$ depend only on $\ell(\lambda)$.
So fixing $n,l \in \bbN$ let us consider the equality
\begin{equation}\label{eq:p_n:pair:ii}
 \sum_{\lambda: \, |\lambda|=n, \, \ell(\lambda)=l}
 t^{2n(\lambda)}
 \begin{bmatrix}l\\ m_1(\lambda), \ldots ,m_n(\lambda) \end{bmatrix}_t
=t^{l(l-1)} \begin{bmatrix} n-1 \\ l-1 \end{bmatrix}_t .
\end{equation}
This equality will deduce 
\begin{align*}
  \left(\text{RHS of \eqref{eq:p_n:pair:i'}}\right)
&=\sum_{l \geq 1} t^{-l(l-1)} u^l \frac{(u^{-1};t)_l}{1-t^l}
   \cdot t^{l(l-1)} \begin{bmatrix} n-1\\ l-1\end{bmatrix}_t
 =\frac{1}{1-t^n} \sum_{l \geq 1}
   \begin{bmatrix} n\\l \end{bmatrix}_t (u^{-1};t)_l u^l
\\
&=\frac{1}{1-t^n}(u^n-1)
 =\left(\text{RHS of \eqref{eq:p_n:pair:i}}\right).
\end{align*}
Here at the third equality we used the elementary formula 
\begin{equation}\label{eq:ex5}
 \sum_{l=0}^n 
 \begin{bmatrix} n \\ l \end{bmatrix}_t (u^{-1};t)_l \, u^l
 = u^n.
\end{equation}
Thus it is enough to show \eqref{eq:p_n:pair:ii}.

The proof is by induction on $n$. 
The case $n=1$ is obvious, so assume $n>1$.
For a partition $\lambda$ such that $\ell(\lambda)=l$,
set $\mu=\wt{\lambda}:=\lambda-(1^l)$.
Then by $\mu'_i = \lambda'_{i+1}$ and $l=l(\lambda)=\lambda'_1$
we have 
\begin{align*}
 2n(\lambda)
&= 2\sum_{i \ge 1} \binom{\lambda'_i}{2}
 = \lambda'_i(\lambda'_i-1)+ 2\sum_{i \ge 1} \binom{\mu'_i}{2} 
 = l(l-1)+2n(\mu).
\end{align*}
We also have  $m_{i+1}(\lambda)=m_{i}(\mu)$, 
and if $\abs{\lambda}=n$ then 
$\abs{\mu}=\abs{\lambda}-l=n-l$.
So we have 
\begin{align*}
  \left(\text{LHS of \eqref{eq:p_n:pair:ii}}\right)
&=\sum_{\lambda: \, |\lambda|=n, \, l(\lambda)=l} t^{l(l-1)}
  \begin{bmatrix} l \\ m_1(\lambda) \end{bmatrix}_t 
  \cdot t^{2n(\wt{\lambda})}
  \begin{bmatrix}
   l-m_1(\lambda) \\ m_2(\lambda),\ldots, m_n(\lambda)
  \end{bmatrix}_t 
\\
&=t^{l(l-1)}
  \sum_{m_1 \in \bbN} 
  \begin{bmatrix} l \\ m_1 \end{bmatrix}_t
  \sum_{\mu: \, |\mu|=n-l, \, l(\mu)=l-m_1} t^{2n(\mu)}
  \begin{bmatrix} 
   l-m_1 \\ m_1(\mu), \ldots, m_{n-l}(\mu)
  \end{bmatrix}_t
\\
&=t^{l(l-1)}\sum_{m_1 \geq 0}t^{(l-m_1)(l-m_1-1)}
  \begin{bmatrix} n-l-1 \\ l-m_1-1 \end{bmatrix}_t
  \begin{bmatrix} l \\ m_1\end{bmatrix}_t
 =t^{l(l-1)} \begin{bmatrix} n-1 \\ l-1 \end{bmatrix}_t
 =\left(\text{RHS of \eqref{eq:p_n:pair:ii}}\right).
\end{align*}
Here at the third equality we used the induction hypothesis,
and at the fourth equality we applied the $q$-Chu-Vandermonde identity
\begin{equation}\label{eq:q-CV}
 \sum_{j \ge 0} t^{(a-k+j)j}
 \begin{bmatrix} a   \\ k-j \end{bmatrix}_t
 \begin{bmatrix} b   \\ j   \end{bmatrix}_t 
=\begin{bmatrix} a+b \\ k   \end{bmatrix}_t
\end{equation}
with $a=l$, $k-j=m_1$, $b=n-l-1$ and  $j=l-m_1-1$.
Thus \eqref{eq:p_n:pair:ii} is shown for any $n$.
\end{proof}

\subsection{Cauchy-type kernel function}

\begin{dfn}\label{dfn:Q_lambda} 
For a partition $\lambda$, define $Q_\lambda \in \Hcl$ by 
\[
 Q_\lambda := 
 P_\lambda / \pair{P_\lambda}{P_\lambda} = 
 P_\lambda \cdot q^{\abs{\lambda}}
 \prod_{i\ge1} (q^{-1};q^{-1})_{m_i(\lambda)}.
\]
\end{dfn}

Let us denote by $\Hcl \wotimes \Hcl$
the completion of $\Hcl \otimes \Hcl$ 
in terms of the $K_0(\catA)=\bbZ$-grading.

\begin{prp}\label{prp:Cauchy}
In $\Hcl \wotimes \Hcl$ we have the following equality.
\[
 \sum_{\lambda:\tpar} P_{\lambda} \otimes Q_{\lambda}
=\exp\Bigl(\sum_{n\ge 1} \frac{1}{n}(q^n-1) p_n \otimes p_n\Bigr).
\]
\end{prp}

\begin{proof}
Using the basis $p_{\lambda}=p_{\lambda_1} * p_{\lambda_2} * \cdots$
in Lemma \ref{lem:p:basis}, 
we have 
\[
 \exp\Bigl(\sum_{n \geq 1}\frac{1}{n}(q^n-1)p_n \otimes p_n\Bigr)
=\sum_{\lambda: \tpar}
 \wt{z}_{\lambda}(q)^{-1}p_{\lambda} \otimes p_{\lambda},
 \quad
 \wt{z}_{\lambda}(q) := 
 \prod_{i\ge1} \Bigl(i^{m_i(\lambda)} \cdot m_i(\lambda_i)!\Bigr) 
 \cdot
 \prod_{i\ge1} (q^{\lambda_i}-1)^{-1}.
\]
By Theorem \ref{thm:p_n:pair} we know that 
$\{p_\lambda \mid \lambda: \tpar\}$ and 
$\{p_\lambda/\wt{z}_{\lambda}(q) \mid \lambda: \tpar\}$
are dual bases of $\Hcl$ with respect to the Hopf pairing.
By Lemma \ref{lem:I_lambda:pair} and Definition \ref{dfn:Q_lambda},
$\{P_\lambda \mid \lambda: \tpar\}$ and
$\{Q_\lambda \mid \lambda: \tpar\}$ are 
also dual bases.
Since the Hopf pairing is non-degenerate, we have 
$\sum_{\lambda: \tpar} P_\lambda \otimes Q_\lambda
=\sum_{\lambda: \tpar} 
  p_\lambda \otimes (p_\lambda/\wt{z}_\lambda(q))$.
\end{proof}

$\sum_\lambda P_\lambda \otimes Q_\lambda$ 
is called the Cauchy-type kernel function in the following sense.
Under the isomorphism 
$\psi: \Hcl \simto \Lambda(x)$
the identity in Proposition \ref{prp:Cauchy} is mapped to 
\[
 \sum_\lambda P_\lambda(x;q^{-1}) Q_\lambda(y;q^{-1})
=\exp\Bigl(\sum_{n\ge1}\frac{1}{n}(1-q^{-n}) p_n(x) p_n(y)\Bigr).
\]
Here $p_n(x)=\sum_i x_i^n \in \Lambda(x)$ denotes 
the power-sum symmetric function, 
$P_\lambda(x;q^{-1}) \in \Lambda(x)$
denotes the Hall-Littlewood symmetric function,
and $Q_\lambda(x;q^{-1})$ denotes the dual basis 
with respect to the inner product 
$\pair{p_n(x)}{p_m(x)}=\delta_{n,m} n/(1-q^{-n})$. 
Taking the limit $q^{-1} \to 0$, 
we have the Cauchy formula \cite[Chap.I \S4 (4.3)]{M} 
for Schur symmetric functions $s_\lambda(x)$:
\[
  \sum_{\lambda: \tpar} s_\lambda(x) s_\lambda(y) 
= \prod_{i,j}(1-x_i y_j)^{-1}
= \exp\Bigl(\sum_{n\ge1}\frac{1}{n}p_n(x) p_n(y)\Bigr).
\]

We close this subsection with Corollary \ref{cor:Q_n} 
of Proposition \ref{prp:Cauchy}.
For its proof we prepare

\begin{lem}\label{lem:P-e}
For any partition $\lambda$ we have
\[
 P_\lambda \in 
 e_{\lambda'} + \sum_{\mu<\lambda} \bbC e_{\mu'},
\]
where $<$ denotes the dominance order of partitions
(see Remark \ref{rmk:dom})
\end{lem}

\begin{proof}
This is a restatement of \eqref{eq:I-X} 
in the proof of Theorem \ref{thm:Hall-Jordan} (2).
\end{proof}

\begin{cor}\label{cor:Q_n}
In $\Hcl[[z]]$ we have 
\[
 \sum_{n\ge0} Q_{(n)} z^n = 
 \exp\Bigl(\sum_{n\ge1}\frac{1}{n}(q^n-1)p_n z^n\Bigr).
\]
\end{cor}

\begin{proof}
Consider the algebra homomorphism 
$\Hcl = \bbC[e_1,e_2,\ldots] \to \bbC[z]$ given by 
$e_n \mapsto \delta_{n,1} z$.
By Lemma \ref{lem:P-e} it yields 
$P_\lambda \mapsto 
 \delta_{\lambda,(\abs{\lambda})} z^{\abs{\lambda}}$, 
and by Definition \ref{dfn:p_n} 
it yields $p_n \mapsto z^n$.
Then applying this homomorphism to the first factor 
of the Cauchy-type kernel function in $\Hcl \wotimes \Hcl$
(Proposition \ref{prp:Cauchy}, we have the desired consequence.
\end{proof}

This statement can be generalized as the following proposition.
Let us set 
\[
 Q(z) := \sum_{n \ge 0} Q_{(n)} z^n \in \Hcl[[z]],
 \quad
 F(z) := \frac{1-q z}{1-z} \in \bbC[[z]].
\]

\begin{prp}[{\cite[Chap.I\!I\!I \S2 (2.15)]{M}}]
\label{prp:Q_lambda}
Let $\lambda=(\lambda_1,\ldots,\lambda_l$ be 
a partition of length $\ell(\lambda)=l$.
For indeterminates $z_1,\ldots,z_l$ we set 
\[
 Q(z_1,\ldots,z_n) := 
 Q(z_1)*\cdots*Q(z_n) \cdot \prod_{i<j} F(z_j/z_i).
\]
Then $Q_\lambda$ is the coefficient of 
$z^\lambda = z_1^{\lambda_1} \cdots z_l^{\lambda_l}$ 
in $Q(z_1,\ldots,z_n)$.
\end{prp}

\section{Derived classical Hall algebra}
\label{sec:dg}

We keep the notations in \S \ref{sec:Hcl} and \S \ref{sec:HL}.
As mentioned in the introduction,
we will realize the differential operators $\partial_{p_n}$,
or the Heisenberg algebra generated by $p_n$ and $\partial_{p_n}$,
in the derived Hall algebra of the category 
$\Rep^{\nil}_{\bbF_q} Q$.

\subsection{Heisenberg relation}
\label{subsec:Heis}

Applying the formalism in \S \ref{subsec:dg} to the dg category of 
nilpotent representations of the Jordan quiver, 
we have the unital associative algebra $\Dcl$ generated by the set 
\[
 \{Z_\lambda^{[n]} \mid n \in \bbZ, \ \lambda: \tpar\}
\]
and the relations
\begin{align}
\label{eq:Toen:1}
 Z_\lambda^{[n]} * Z_\mu^{[n]} &= 
 \sum_{\nu: \tpar} g^{\nu}_{\lambda,\mu} Z_\lambda^{[n]}
\\
\label{eq:Toen:2}
 Z_\lambda^{[n]} * Z_\mu^{[n+1]} &= 
 \sum_{\alpha,\beta: \tpar} 
 \gamma^{\alpha,\beta}_{\lambda,\mu} 
 Z_\alpha^{[n+1]} * Z_\beta^{[n]},
\\
\nonumber
 Z_\lambda^{[n]} * Z_\mu^{[m]} &= Z_\mu^{[m]} * Z_\lambda^{[n]},
 \quad (n-m<-1).
\end{align}
Here 
$g^\nu_{\lambda,\mu}=
\abs{\{M \subset I_\nu \mid M \simeq I_{\mu}, I_\nu/M \simeq I_\lambda\}}$ 
as in \eqref{eq:Hall(A)} of \S \ref{subsec:Jordan},
and 
\[
 \gamma^{\alpha,\beta}_{\lambda,\mu} = 
 \frac{\abs{\{0 \to I_\alpha \to I_\lambda \to I_\beta \to I_\nu \to 0 
              \mid \text{exact in $\catA$}\}}}{a_\lambda a_\mu}.
\]
We call $\Dcl$ the derived classical Hall algebra.

Now recalling the primitive elements of $\Hcl$ in Definition \ref{dfn:p_n} 
and Theorem \ref{thm:prim}, we consider

\begin{dfn}
For $n \in \bbZ_{>0}$ we define $b_{\pm n} \in \Dcl$ by  
\[
 b_n := 
 \sum_{\abs{\lambda}=n}(q;q)_{\ell(\lambda)-1} Z_\lambda^{[0]},
 \quad
 b_{-n} := 
 \sum_{\abs{\lambda}=n}(q;q)_{\ell(\lambda)-1} Z_\lambda^{[1]}.
\]
We also set $b_0 := 1 \in \Dcl$.
($b$ reads ``boson".)
\end{dfn}

\begin{thm}[Heisenberg relation]\label{thm:Heis}
For $m,n \in \bbZ$ we have 
\[
 b_m * b_n - b_n * b_m = \delta_{m+n,0}\frac{m}{q^m-1}.
\]
Precisely speaking, for $m,n$ with $m n \neq 0$ 
we have the above formula,
and for $m,n$ with $m n = 0$ we have $b_m * b_n - b_n * b_m = 0$.
\end{thm}

Note that the right hand side is equal to the Hopf pairing
$\pair{p_m}{p_n}$ in $\Hcl$ (Theorem \ref{thm:p_n:pair}).

\begin{proof}
By the relation \eqref{eq:Toen:1} and commutativity 
$g^{\lambda}_{\mu,\nu}=g^{\lambda}_{\nu,\mu}$ of $\Hcl$,
we have $b_m * b_n =b_n * b_m$ for $m n \ge 0$.
So let us assume $m,n \in \bbZ_{>0}$ and consider 
$b_m * b_{-n} - b_{-n} * b_m$.

Set $c_\lambda := (q;q)_{\ell(\lambda)-1}$.
Then 
$\Delta(p_l)=p_l \otimes 1 + 1 \otimes p_l$ 
for $p_l = \sum_{\abs{\lambda}=l} c_\lambda [I_\lambda] \in \Hcl$ 
is equivalent to 
\begin{align}\label{eq:prim}
 \begin{cases} 
 \sum_{\lambda: \tpar} 
 c_\lambda e^{\lambda}_{\mu,\nu} a_\lambda^{-1} = 0 
 & (\mu,\nu \neq \emptyset) \\
 e^{\lambda}_{\mu,\emptyset} = e^{\lambda}_{\emptyset,\mu}
                             = \delta_{\lambda,\mu} a_\lambda 
 \end{cases}.
\end{align}
Here we used 
$e^\lambda_{\mu,\nu}
=\abs{\{0 \to I_\nu \to I_\lambda \to I_\mu \to 0 \mid 
        \text{exact in $\catA$}\}}
=g^\lambda_{\mu,\nu} a_\mu a_\nu$.
Recall also that $e^\lambda_{\mu,\nu}=0$ 
if $\abs{\lambda} \neq \abs{\mu}+\abs{\nu}$ by Lemma \ref{lem:graded}.
Now we note that 
\begin{align}\label{eq:gamma}
 \gamma^{\alpha,\beta}_{\mu,\nu} = 
 \sum_{\nu: \tpar}
 \frac{e^{\mu}_{\lambda,\alpha}e^{\nu}_{\beta,\lambda}}
      {a_\lambda a_\mu a_\nu}.
\end{align}
Then by the relation \eqref{eq:Toen:2} we have
\begin{align*}
  a_m * a_{-n} 
&=\sum_{\abs{\mu}=m} \sum_{\abs{\nu}=n} 
  c_\mu c_\nu Z_\mu^{[0]} * Z_\nu^{[1]} 
=\sum_{\abs{\mu}=m} \sum_{\abs{\nu}=n} c_\mu c_\nu 
  \sum_{\substack{\alpha,\beta: \\ \tpar}}
  \gamma^{\alpha \beta}_{\mu, \nu} Z_\alpha^{[1]} * Z_\beta^{[0]}.
\end{align*}
Fix  partitions $\alpha$ and  $\beta$.
Then using \eqref{eq:gamma} we can compute 
the coefficient of $Z_\alpha^{[1]} Z_\beta^{[0]}$ as 
\begin{align*}
  \sum_{\abs{\mu}=m} \sum_{\abs{\nu}=n} 
  c_\mu c_\nu \gamma^{\alpha \beta}_{\mu, \nu} 
&=\sum_{\abs{\mu}=m} \sum_{\abs{\nu}=n} \sum_{\lambda: \tpar}
  c_\mu c_\nu 
  \frac{e^{\mu}_{\lambda,\alpha}e^{\nu}_{\beta,\lambda}}
       {a_\lambda a_\mu a_\nu}
 =\sum_{\substack{\lambda: \tpar \\ \abs{\mu}=m}}
  c_\mu e^{\mu}_{\lambda,\alpha} a_\lambda^{-1} a_\mu^{-1}
  \sum_{\abs{\nu}=n} c_\nu e^\nu_{\beta,\lambda}a_\nu^{-1}.
\end{align*}
Then using \eqref{eq:prim} repeatedly we proceed as
\begin{align*}
&=\sum_{\substack{\lambda: \tpar \\ \abs{\mu}=m}} 
  c_\mu e^{\mu}_{\lambda,\alpha} a_\lambda^{-1} a_\mu^{-1}
  \bigl(\delta_{\beta,\emptyset} c_\lambda 
      + \delta_{\lambda, \emptyset}\delta_{\abs{\beta},n} c_\beta \bigr)
\\
&=\delta_{\beta,\emptyset} 
  \sum_{\lambda: \tpar} c_\lambda a_\lambda^{-1}
  \sum_{\abs{\mu}=m} c_\mu e^\mu_{\lambda,\alpha} a_\mu^{-1}
 +\delta_{\abs{\beta},n}c_\beta 
  \sum_{\abs{\mu}=m} c_\mu e^\mu_{\emptyset,\alpha} a_\mu^{-1}
\\
&=\delta_{\beta,\emptyset} 
  \sum_{\abs{\lambda}=m} c_\lambda a_\lambda^{-1}
  \delta_{\alpha,\emptyset}\delta_{m,n}c_\lambda 
 +\delta_{\abs{\alpha},m} \delta_{\abs{\beta},n}c_\alpha c_\beta
\\
&=
  \delta_{\alpha,\emptyset} \delta_{\beta,\emptyset} \delta_{m,n}
  \sum_{\abs{\lambda}=m} c_\lambda^2 a_\lambda^{-1}
 +\delta_{\abs{\alpha},m} \delta_{\abs{\beta},n}c_\alpha c_\beta.
\end{align*}
Summing up on $\alpha$ and $\beta$, we have the desired result:
\[
 b_m * b_{-n} = \delta_{m,n} \pair{p_m}{p_m} + b_{-n} * b_m. 
\]
\end{proof}

Thus we realized the differential operator 
$\partial_{p_n}$ of $p_n \in \Hcl$ 
by the embedding 
\[
 \Hcl \longinj \Dcl, \quad 
 [I_\lambda] \longmapsto Z_\lambda^{[1]}, \quad 
 p_n \longmapsto b_{-n},
\]
and $\partial_{p_n} =  b_n/\pair{p_n}{p_n}$.

\subsection{Hall-Littlewood element as eigenfunction}
\label{subsec:eigen}

We continue to use the embedding 
$\Hcl \inj \Dcl$, $[I_\lambda] \mapsto Z_\lambda^{[1]}$,
and write the image of elements in $\Hcl$ by the same letter.
In view of  Theorem \ref{thm:Heis},we also use the symbol 
\[
 \partial_{p_n} := n(q^n-1)^{-1} b_n \in \Dcl.
\]
Let us denote by $\wDcl$ the completion of $\Dcl$ 
with respect to the grading by $K_0(\catA)=\bbZ$.

\begin{prp}
Define $D(z) \in \wDcl [[z^{\pm1}]]$ by 
\begin{align*}
   D(z) 
:=&\exp \Bigl(-\sum_{n\ge 1}\frac{1}{n}(q^n-1)b_{-n}z^n\Bigr) * 
   \exp \Bigl(-\sum_{n\ge 1}\frac{1}{n}(q^n-1)b_n z^{-n}\Bigr)
\\
 =&\exp \Bigl(-\sum_{n\ge 1}\frac{1}{n}(q^n-1)p_n z^n\Bigr) * 
   \exp \Bigl(-\sum_{n\ge 1}\partial_{p_n} z^{-n} \Bigr)
\end{align*}
with  $\exp(x):=\sum_{n\ge0}(x*\cdots*x)/n!$.
Define $D_n \in \wDcl$ by the expansion 
$D(z) = \sum_{n\in\bbZ} D_n z^n$.
Then, regarding $D_0$ as an operator acting on $\Hcl$, 
we have 
 \[
 D_0\left(P_\lambda\right) = q^{\ell(\lambda)} P_{\lambda}.
\]
\end{prp}

\begin{proof}
Divide $D(z)=R(z) * R^\perp(z)$ with 
\begin{align*}
R(z)&:=\exp\Bigl(-\sum_{n \ge 1}\frac{1}{n}(q^n-1)b_{-n}z^n \Bigr)
      =\exp\Bigl(-\sum_{n \ge 1}\frac{1}{n}(q^n-1)p_n   z^n \Bigr), 
\\
R^\perp(z)
    &:=\exp\Bigl(-\sum_{n \ge 1}\frac{1}{n}(q^n-1)b_n z^{-n} \Bigr) 
     =\exp(-\sum_{n \ge 1}\partial_{p_n}z^{-n}). 
\end{align*}
We also set
\[
 Q(z) :=\exp\Bigl(\sum_{n \ge 1}\frac{1}{n}(q^n-1)p_n z^n \Bigr)
       =R^{-1}(z).
\]
By Corollary \ref{cor:Q_n} we have
\begin{align}\label{eq:Q(z)}
 Q(z) = \sum_{n\ge0} Q_{(n)} z^n.
\end{align}
By the Baker-Campbell-Hausdorff formula 
we have the commutation relation 
\[
 R^\perp(z) * Q(w) = Q(w) * R^\perp(z) F(w/z),
\]
where $F(z) \in \bbC[[z]]$ denotes 
\[
 F(z) := \exp\Bigl(-\sum_{n \ge 1} \frac{1}{n}(q^n-1)z^n\Bigr)
       = \frac{1-q z}{1-z}.
\]
Thus we have
\begin{align}\label{eq:DQ}
  D(z) * Q(w) = Q(w) * D(z) F(z) 
= Q(w) * R(z) * R^\perp(z) \frac{1-q w/z}{1-w/z}.
\end{align}

Now we show the statement for $\ell(\lambda)=1$,
i.e., the case $\lambda=(n)$.
Let us denote by $\Res_{z=0}\bigl(f(z)\bigr)$ 
the coefficient of $z^{-1}$ in the Laurent series $f(z)$.
Then from \eqref{eq:DQ} we have 
\begin{align*}
   D_0\left(Q(w)\right)
&= \Res_{z=0}\Bigl(z^{-1} Q(w) * R(z) \cdot \frac{1-q w/z}{1-w/z} \Bigr)
\\
&= Q(w) * R(0) \cdot q +w^{-1} Q(w) * R(w) \cdot (w-q w)
 = q Q(w)+(1-q).
\end{align*}
Thus from \eqref{eq:Q(z)} we have $D_0(Q_{(n)}) = q Q_{(n)}$.

Next we consider the case $\lambda=(\lambda_1,\ldots,\lambda_l)$ 
with $\ell(\lambda)=l>1$. 
Recall Proposition \ref{prp:Q_lambda} 
which states that $Q_\lambda$ is the coefficient of 
$w^\lambda = w_1^{\lambda_1} \cdots w_l^{\lambda_l}$ 
in
\[
 Q(w_1,\ldots,w_l) := Q(w_1) * \cdots * Q(w_l) \cdot
 \prod_{i<j}F(w_j/w_i).
\]
Then using \eqref{eq:DQ} we compute
\begin{align*}
  D(z) * Q(w_1,\ldots,w_l)
= Q(w_1,\ldots,w_l) * D(z) * \prod_{i=1}^l F(w_i/z). 
\end{align*}
Thus
\begin{align*}
  D_0\left(Q(w_1,\ldots,w_l)\right)
&=\Res_{z=0}\Bigl(z^{-1}Q(w_1,\ldots,w_l) * R(z) \cdot 
                  \prod_{i=1}^l \frac{1-q w_i/z}{1-w_i/z} \Bigr)
\\
&=Q(w_1,\ldots,w_l) * R(0) \cdot q^l + 
  \sum_{i=1}^l w_i^{-1} Q(w_1,\ldots,w_l) * R(w_i) 
  \cdot (w_i-q w_i) \cdot \prod_{j \neq i}F(w_i/w_j)
\\
&=q^l Q(w_1,\ldots,w_l)  + 
  (1-q) \sum_{i=1}^l  Q(w_1,\ldots,\wh{w}_i,\ldots,w_l).
\end{align*}
Taking the coefficients of $w^\lambda$,
we have none from the summation in the right hand side
and obtain the statement.
\end{proof}

\subsection{Jing's vertex operator}

We close this note by recalling Jing's vertex operator \cite{J}
whose composition reconstructs Hall-Littlewood polynomials.

\begin{prp}
Define $B(z) \in \wDcl [[z^{\pm1}]]$ by 
\begin{align*}
   B(z) 
:=&\exp \Bigl( \sum_{n\ge 1}\frac{1}{n}(q^n-1)b_{-n}z^n\Bigr) * 
   \exp \Bigl(-\sum_{n\ge 1}\frac{1}{n}(q^n-1)b_n z^{-n}\Bigr)
\\
 =&\exp \Bigl( \sum_{n\ge 1}\frac{1}{n}(q^n-1)p_n z^n\Bigr) * 
   \exp \Bigl(-\sum_{n\ge 1}\partial_{p_n} z^{-n} \Bigr)
\end{align*}
and expand it as $B(z) = \sum_{n\in\bbZ} B_n z^n$.
Then, regarding $B_n$ as an operator acting on $\Hcl$, 
for any partition $\lambda=(\lambda_1,\ldots,\lambda_l)$ 
of length $\ell(\lambda)=l$ we have
\[
 Q_\lambda = \left(B_{\lambda_1} * \cdots * B_{\lambda_l}\right)(1).
\]
\end{prp} 

The proof is the same as in \cite{J} and  
\cite[Chap.I\!I\!I \S5 Exercise 8]{M}
so we omit it.


\end{document}